\documentclass[
  11pt,
  reqno
]{amsart}

\usepackage{mathtools}
\usepackage{amssymb}
\usepackage{amsthm}
\usepackage{mathrsfs}

\usepackage[T1]{fontenc}
\usepackage{lmodern}

\usepackage{microtype}
\usepackage{enumitem}

\usepackage[
  backend=biber,
  style=numeric,
  doi=true,
  sorting=nyt,
  eprint=true,
  giveninits=true,
  maxnames=99
]{biblatex}
\DeclareFieldFormat[article]{title}{\mkbibemph{#1}}
\renewbibmacro*{in:}{%
  \ifentrytype{article}
    {}
    {\printtext{\bibstring{in}\intitlepunct}}}

\usepackage[
  colorlinks=false,
  linkbordercolor={1 0 0},
  citebordercolor={0 1 0},
  urlbordercolor={0 1 1}
]{hyperref}
\usepackage[nameinlink,noabbrev]{cleveref}

\newcommand{\Z}{\mathbb{Z}}

\DeclarePairedDelimiter{\abs}{\lvert}{\rvert}

\DeclareMathOperator{\card}{Card}

\DeclareMathOperator{\diam}{diam}

\newcommand{\zz}{\Z}

\newcommand{\yoemph}[1]{\emph{#1}}

\newcommand{\yoch}{\mathrm{CH}}

\newcommand{\yodis}{Z}
\newcommand{\yocodesp}{Z_0}

\newcommand{\yoomega}{\omega_{0}}

\newcommand{\yoopenball}[3]{\mathrm{U}(#1,#2;#3)}
\newcommand{\yoclosedball}[3]{\mathrm{B}(#1,#2;#3)}
\newcommand{\yobard}{[\![d]\!]}

\newcommand{\yocantor}{{}^{\yoomega}2}

\newcommand{\yocpm}[1]{\mathrm{CPM}(#1)}
\newcommand{\yometsp}[1]{\mathrm{Met}(#1)}
\newcommand{\yocomp}[1]{\mathrm{Comp}(#1)}
\newcommand{\yometbd}[1]{\mathrm{BMet}(#1)}
\newcommand{\yocpmbd}[1]{\mathrm{BCPM}(#1)}

\newcommand{\yometdis}[1]{\mathcal{D}_{#1}}

\newcommand{\youcpm}[2]{\mathrm{UCPM}(#1; #2)}
\newcommand{\youltsp}[2]{\mathrm{UMet}(#1;#2)}
\newcommand{\youcompm}[2]{\mathrm{UComp}(#1; #2)}

\newcommand{\youclnet}{A}
\newcommand{\youquotpt}{\varpi_{\youclnet}}
\newcommand{\youeta}{\eta}

\DeclareMathOperator{\youmetdis}{\mathcal{UD}}
\DeclareMathOperator{\youadm}{Adm}

\newcommand{\yoslameto}{\mathbf{M}_{<\infty}}
\newcommand{\yoslamet}[1]{\yoslameto(#1)}

\newcommand{\youslameto}{\mathbf{UM}_{<\infty}}
\newcommand{\youslamet}[2]{\youslameto(#1;#2)}

\newcommand{\yoccl}{\mathfrak{C}}
\newcommand{\yocclq}[1]{\yoccl(#1)}
\newcommand{\yopd}[1]{\mathfrak{P}_{#1}}
\newcommand{\youccl}{\mathfrak{C}^{R}}
\newcommand{\youcclq}[1]{\youccl(#1)}

\newcommand{\yocase}[2]{\medskip\noindent\textbf{Case #1.} #2.\par}

\newcommand{\yoconti}{\mathfrak{c}}

\numberwithin{equation}{section}

\theoremstyle{plain}
\newtheorem{theorem}{Theorem}[section]
\newtheorem{lemma}[theorem]{Lemma}
\newtheorem{proposition}[theorem]{Proposition}
\newtheorem{corollary}[theorem]{Corollary}

\theoremstyle{definition}
\newtheorem{definition}[theorem]{Definition}

\theoremstyle{remark}

\newenvironment{acknowledgements}{%
  \medskip\noindent\textit{Acknowledgements.}\ }{\par}
\newenvironment{useofai}{%
  \medskip\noindent\textbf{Use of AI.}\ }{\par}

\crefname{theorem}{theorem}{theorems}
\Crefname{theorem}{Theorem}{Theorems}

\crefname{lemma}{lemma}{lemmas}
\Crefname{lemma}{Lemma}{Lemmas}

\crefname{proposition}{proposition}{propositions}
\Crefname{proposition}{Proposition}{Propositions}

\crefname{corollary}{corollary}{corollaries}
\Crefname{corollary}{Corollary}{Corollaries}

\crefname{claim}{claim}{claims}
\Crefname{claim}{Claim}{Claims}

\crefname{definition}{definition}{definitions}
\Crefname{definition}{Definition}{Definitions}

\crefname{example}{example}{examples}
\Crefname{example}{Example}{Examples}

\crefname{question}{question}{questions}
\Crefname{question}{Question}{Questions}

\crefname{remark}{remark}{remarks}
\Crefname{remark}{Remark}{Remarks}

\AddToHook{env/theorem/begin}{\crefalias{section}{theorem}\crefalias{theorem}{theorem}}
\AddToHook{env/lemma/begin}{\crefalias{section}{lemma}\crefalias{theorem}{lemma}}
\AddToHook{env/proposition/begin}{\crefalias{section}{proposition}\crefalias{theorem}{proposition}}
\AddToHook{env/corollary/begin}{\crefalias{section}{corollary}\crefalias{theorem}{corollary}}
\AddToHook{env/claim/begin}{\crefalias{section}{claim}\crefalias{theorem}{claim}}
\AddToHook{env/definition/begin}{\crefalias{section}{definition}\crefalias{theorem}{definition}}
\AddToHook{env/example/begin}{\crefalias{section}{example}\crefalias{theorem}{example}}
\AddToHook{env/question/begin}{\crefalias{section}{question}\crefalias{theorem}{question}}
\AddToHook{env/remark/begin}{\crefalias{section}{remark}\crefalias{theorem}{remark}}

\title[Borelness of moduli spaces of metrics]{
  Borelness of Moduli Spaces of Metrics Implies Separability}

\author{Yoshito Ishiki}
\address{Department of Mathematical Sciences\\
Tokyo Metropolitan University\\
Minami-osawa, Hachioji, Tokyo 192-0397, Japan}
\email{ishiki-yoshito@tmu.ac.jp}

\author{Tomoki Uda}
\address{Department of Electronics and Communication Technology\\
Faculty of Science and Technology\\
Nanzan University\\
18 Yamazato-cho, Showa-ku, Nagoya 466-8673, Japan}
\email{uda0@nanzan-u.ac.jp}

\date{November 3, 2026}



\hypersetup{
  pdftitle={Borelness of Moduli Spaces of Metrics Implies Separability},
  pdfauthor={Yoshito Ishiki and Tomoki Uda},
  pdfsubject={Borelness and complete metrizability of spaces of metrics},
  pdfkeywords={spaces of metrics, bounded metrics, descriptive set theory,
    Borel sets, complete metrizability}
}

\begin{document}

\begin{abstract}
Let $X$ be a metrizable space, and let $\mathrm{Met}(X)$ denote the space of metrics compatible with the topology of $X$, regarded as a subspace of the space of continuous pseudometrics with the topology induced by the supremum metric. We first prove that if $Z$ is a discrete space whose cardinality is the first uncountable cardinal, then $\mathrm{Met}(Z)$ is not Borel. As a consequence, if $\mathrm{Met}(X)$ is Borel, then $X$ is separable. Combined with a theorem of Koshino, our result yields that the space of bounded compatible metrics on a metrizable space $X$ is completely metrizable if and only if $X$ is $\sigma$-compact. We also establish non-Archimedean analogues for spaces of ultrametrics. We further prove componentwise versions of these results, where two metrics belong to the same component if their distance is finite.
\end{abstract}

\maketitle


\section{Introduction}

Let
$X$
be a metrizable space.
We write
$\yometsp{X}$
for the space of all
compatible metrics and
$\yometbd{X}$
for its subspace of bounded metrics.
They are regarded as subspaces of the space
$\yocpm{X}$
of continuous
pseudometrics,
equipped with
the topology induced by the supremum metric.

Koshino~\cite{zbMATH07960706}
studied
the Borel complexity of the space
$\yometbd{X}$
 of bounded compatible metrics,
 which  is denoted by
$\mathrm{AM}(X)$
 in his paper.
In particular,
for a separable
metrizable space
$X$,
he proved that
$X$
is
$\sigma$-compact
if and only if
$\yometbd{X}$
 is completely metrizable
\cite[Corollary~1.2]{zbMATH07960706}.

To remove the separability assumption of this characterization,
the first author asked whether
$\yometsp{\yodis}$
 is not completely metrizable when
 $\yodis$
 is a discrete space of
cardinality
$\aleph_{1}$
\cite[Conjecture~5.1]{Ishiki2024IsometricExtensor}.

We answer this question affirmatively.
We first prove that if
$\yodis$
is a discrete space of cardinality
$\aleph_1$,
then
$\yometsp{\yodis}$
 is not Borel
(\Cref{cor:discrete-metric-nonborel}).
As a consequence,
for every metrizable space
$X$,
if
$\yometsp{X}$
is Borel,
then
$X$
is separable
(\Cref{cor:nonseparable-metric-nonborel}).
Combined with a theorem of Koshino,
our result yields that,
for every metrizable space
$X$
that is not necessarily separable,
the space
$\yometbd{X}$
is
completely metrizable
if and only if
$X$
is
$\sigma$-compact
(\Cref{cor:bounded-metric-complete-sigma-compact}).

The proof of our main theorems
begins with a dense
non-Borel subset
 $S$
 of
 $[0,1]$
 of
cardinality
$\aleph_{1}$.
For a discrete space
$\yodis$
of cardinality
$\aleph_{1}$,
we construct a continuous map
$\Phi\colon[0,1]\to\yocpm{\yodis}$
such that
\[
\Phi^{-1}(\yometbd{\yodis})=[0,1]\setminus S.
\]
Since this complement is non-Borel,
we conclude that
$\yometbd{{\yodis}}$
is
non-Borel.
For a closed discrete subspace
$A$
of
$X$,
we use the first author's operator
\cite{Ishiki2024IsometricExtensor}
\[
E\colon\yometsp{A}\to\yometsp{X}.
\]
For every
$d\in\yometsp{A}$,
we have
$E(d)|_{A^2}=d$,
and
$E$
preserves supremum distances between metrics.
We use its continuous extension from
$\yocpm{A}$
to
$\yocpm{X}$
to transfer non-Borelness from the closed discrete subspace to
$X$
by taking inverse images.

We also show that if
$X$
is completely metrizable and the space of complete compatible metrics on
$X$
is Borel in
$\yocpm{X}$,
then
$X$
is separable
(\Cref{cor:complete-metric-borel-separability}).
Replacing
$[0, 1]$
with
a
Cantor ultrametric space,
we also obtain
a non-Archimedean version of our main results
(\Cref{thm:discrete-ultrametric-nonborel,thm:ultrametric-borel-separability,cor:complete-ultrametric-borel-separability}).

\medskip
\noindent\textbf{Organization.}
In \Cref{sec:preliminaries},
we recall the spaces of continuous pseudometrics
and their components,
together with the extension and Borel pullback results needed to transfer
non-Borelness from a discrete subspace.
In \Cref{sec:archimedean-results},
we construct topological embeddings of
$[0,1]$
into spaces of continuous pseudometrics on discrete spaces and use extension
operators to prove non-Borelness of prescribed components on nonseparable
metrizable spaces.
These results imply that Borelness of
$\yometsp{X}$
forces
$X$
to be separable.
We also characterize complete metrizability of
$\yometsp{X}$
and its components in terms of
$\sigma$-compactness of
$X$.
For completely metrizable spaces,
we further prove that Borelness of the space of complete compatible metrics
implies separability of the underlying space.
In \Cref{sec:non-archimedean-results},
we use Cantor space
$\yocantor$
instead of
$[0,1]$
and apply ultrametric extensors to prove non-Borelness of spaces of compatible
ultrametrics with a prescribed range and of their components on nonseparable
ultrametrizable spaces.
For the space of complete compatible ultrametrics,
we prove non-Borelness whenever it is nonempty and the underlying space is nonseparable.

\medskip
\noindent\textbf{Conventions and notation.}
We write
$\omega_0=\Z_{\ge 0}$
for the first infinite ordinal,
$\omega_1$
for the first uncountable ordinal.
For a set
$A$,
we denote its cardinality by
$\card(A)$.
The supremum of the empty family of nonnegative numbers is taken to be
$0$.

\begin{useofai}
OpenAI Codex was used in the preparation of this manuscript for language
editing,
\LaTeX{}
typesetting assistance,
and exploration of proof
constructions.
 In particular,
 Codex suggested the initial construction
in \Cref{lem:metric-coding},
which the authors then verified and developed
into the related constructions in this paper.
The authors take full
responsibility for the mathematical content and the final version of the
manuscript.
\end{useofai}

\begin{acknowledgements}
  The second author wrote
  an earlier version
  \cite{uda2026ishikisconjecturemathrmmetdcompletely}
  of the present  paper,
  and the first author subsequently expanded and refined its results.
The first author  was supported by JSPS KAKENHI Grant Number JP24KJ0182.
\end{acknowledgements}

\section{Preliminaries}\label{sec:preliminaries}
We study the Borelness of spaces of compatible metrics as subsets of spaces of continuous pseudometrics.
We recall extension operators that extend metrics from a closed subspace to the whole space and preserve distances between metrics.
We also construct dense non-Borel sets of cardinality
$\aleph_1$
in
$[0,1]$
and Cantor space
$\yocantor$
for the proofs in the following sections.

For a set
$X$,
a pseudometric
$d\colon X^2\to[0,\infty]$,
a point
$x\in X$,
and
$r>0$,
we write
\[
\yoopenball{x}{r}{d}\coloneqq\{y\in X\mid d(x,y)<r\},\qquad
\yoclosedball{x}{r}{d}\coloneqq\{y\in X\mid d(x,y)\le r\}
\]
for the open and closed balls,
respectively.

For nonnegative real numbers
$a$
and
$b$,
we write
$a\lor b=\max\{a,b\}$.

We first define the space of continuous pseudo-metrics as the ambient space in which Borelness will be considered.

\begin{definition}\label{def:continuous-pseudometric-space}
Let
$X$
be a topological space.
A map
$d\colon X\times X\to [0,\infty)$
is called a
\yoemph{continuous pseudometric}
if it is continuous and if
for all
$x,y,z\in X$
 we have
\[
d(x,x)=0,
\qquad
d(x,y)=d(y,x),
\qquad
d(x,z)\le d(x,y)+d(y,z).
\]
We denote by
$\yocpm{X}$
the set of all continuous pseudometrics on
$X$.
It is equipped with the topology induced by the
extended supremum metric defined,
for
$d,e\in\yocpm{X}$,
by
\[
\yometdis{X}(d,e)
=\sup_{(x,y)\in X^2}\abs{d(x,y)-e(x,y)}.
\]
The value
$\infty$
is allowed.
Equivalently,
this topology is induced by the metric
$\min\{\yometdis{X},1\}$.
For a metrizable space
$X$,
we denote by
$\yometsp{X}$
the subspace of metrics inducing the topology of
$X$,
and by
$\yometbd{X}$
its subspace of bounded metrics.
The subspace of all bounded members of
$\yocpm{X}$
 is denoted by
$\yocpmbd{X}$.
\end{definition}

The following result establishes a relationship between bounded compatible metrics and
$\sigma$-compactness of the underlying space.

\begin{theorem}[{\cite[Corollary~1.2]{zbMATH07960706}}]\label{thm:koshino-characterization}
Let
$X$
be a separable metrizable space.
Then
$\yometbd{X}$
is completely metrizable
if and only if
$X$
is
$\sigma$-compact.
\end{theorem}

For a metrizable space
$X$,
we denote by
$\yocomp{X}$
the set of all complete metrics in
$\yometsp{X}$.

The following extension preserves uniform distances between metrics and
extends continuously to the ambient spaces of pseudometrics.

\begin{theorem}[{\cite[Theorem~1.1]{Ishiki2024IsometricExtensor}}]
\label{thm:metric-isometric-extensor}
Let
 $X$
be a metrizable space and let
$A$
be a closed subset of
$X$.
Then there exist isometric embeddings
\[
E\colon\yometsp{A}\to\yometsp{X}
\]
and
$\widetilde E\colon\yocpm{A}\to\yocpm{X}$
such that the following statements hold.
\begin{enumerate}[label=\textup{(E\arabic*)},ref=\textup{(E\arabic*)}]
\item\label{item:metric-extensor-extension}
$\widetilde E$
extends
$E$,
and
for every
$d\in\yometsp{A}$,
we have
$E(d)|_{A^2}=d$.

\item\label{item:metric-extensor-isometry}
 For all
 $d,e\in\yometsp{A}$,
 we have
\[
\yometdis{X}(E(d),E(e))
=\yometdis{A}(d,e).
\]

\item\label{item:metric-extensor-bounded-closed-image}
We have
$E(\yometbd{A})\subseteq\yometbd{X}$
and
\[
\widetilde E(\yometsp{A})
=\widetilde E(\yocpm{A})\cap\yometsp{X}.
\]
\item\label{item:metric-extensor-completeness}
If
$X$
is completely metrizable,
then the embeddings can be chosen so that
$E(\yocomp{A})\subseteq\yocomp{X}$.
\end{enumerate}
In particular,
$E(\yometsp{A})$
is closed in
$\yometsp{X}$.
\end{theorem}

{We use the following lemma to join two metrics at one point when extending our constructions to larger discrete spaces.}

\begin{lemma}[One-point amalgamation]\label{lem:one-point-amalgamation}
Let
$X_0$
and
$X_1$
be sets such that
$X_0\cap X_1=\{z\}$.
For
$i\in\{0,1\}$,
let
$d_i$
be a metric on
$X_i$.
Define a symmetric function
$d$
on
$(X_0\cup X_1)^2$
 by
\[
d(x,y)=
\begin{cases}
d_i(x,y) & x,y\in X_i,\\
d_0(x,z)+d_1(z,y)
  & x\in X_0\setminus\{z\},\ y\in X_1\setminus\{z\}.
\end{cases}
\]
Then
$d$
is a metric.
If
$d_0$
and
$d_1$
are complete,
then
$d$
is complete.

If
$u_i$
is an ultrametric on
$X_i$
for each
$i\in\{0,1\}$,
then the symmetric function
$u$
on
$(X_0\cup X_1)^2$
defined by
\[
u(x,y)=
\begin{cases}
u_i(x,y) & x,y\in X_i,\\
u_0(x,z)\lor u_1(z,y)
  & x\in X_0\setminus\{z\},\ y\in X_1\setminus\{z\}.
\end{cases}
\]
is an ultrametric.
If
$u_0$
and
$u_1$
are complete,
then
$u$
is complete.
\end{lemma}

\begin{proof}
The metric assertion is the one-point case of
\cite[Proposition~3.2]{zbMATH08213674}.
For the ultrametric assertion,
apply
\cite[Lemma~2.15]{Ishiki2023FractalDimensions}
to disjoint copies of
$X_0$
and
$X_1$.
\end{proof}

We use nets and separated sets to transfer results from discrete spaces to general metrizable spaces.

\begin{definition}\label{def:separated-net}
Let
$(X,d)$
be a metric space and let
$\varepsilon>0$.
A subset
$A$
of
$X$
is called an
$\varepsilon$-net
in
$(X,d)$
if,
for every
$x\in X$,
there exists
$a\in A$
such that
$d(x,a)<\varepsilon$.
A subset
$A$
of
$X$
is
$\varepsilon$-separated
if
$d(a,b)\ge\varepsilon$
for all distinct
$a,b\in A$.
\end{definition}

\begin{lemma}\label{lem:maximal-separated-set}
Every nonseparable metric space
$(X,d)$
contains a closed discrete subspace
of cardinality
$\aleph_1$.
More precisely,
there is an
$\varepsilon>0$
for which
$(X,d)$
has no
countable
$\varepsilon$-net,
and every maximal
$\varepsilon$-separated
subset
of
$(X,d)$
is an uncountable closed discrete
$\varepsilon$-net.
\end{lemma}

\begin{lemma}\label{lem:maximal-separated-set-ultrametric}
Let
$(X,d)$
be an ultrametric space,
let
$r\in (0, \infty)$,
and let
$A$
be a maximal
$r$-separated
subset of
$(X,d)$.
Then the open balls of radius
$r$
centered at the points of
$A$
form a clopen
partition of
$X$.
{Define
$\pi\colon X\to A$
by requiring that,
for each
$x\in X$
and the unique point
$a\in A$
with
$x\in\yoopenball{a}{r}{d}$,
\[
\pi(x)=a.
\]
Then
$\pi$
is a continuous retraction and,
for every
$x\in X$,
satisfies
\[
d(x,\pi(x))<r.
\]}
\end{lemma}

\begin{proof}
{By maximality of
$A$,
the family
$\{\yoopenball{a}{r}{d}\}_{a\in A}$
covers
$X$.
For distinct
$a,b\in A$,
the inequality
$d(a,b)\ge r$
and the strong triangle inequality imply
$\yoopenball{a}{r}{d}\cap\yoopenball{b}{r}{d}=\emptyset$.
}
The family
$\{\yoopenball{a}{r}{d}\}_{a\in A}$
is a partition of
$X$.
Since each ball is clopen,
the map
$\pi$
is continuous.
Its definition implies
$\pi(a)=a$
for every
$a\in A$
and
$d(x,\pi(x))<r$
for every
$x\in X$.
\end{proof}

\subsection{Completeness of spaces of pseudometrics}
We use completeness of the ambient spaces to extend isometric maps from dense
subsets of Cantor space to the whole Cantor space.
We also use completeness of the ambient components to prove complete
metrizability of their
$G_\delta$
subspaces.

\begin{lemma}
\label{lem:cpm-complete}
Let
$X$
be a topological space.
Then every
sequence in
$\yocpm{X}$
that is Cauchy with respect to
$\yometdis{X}$
converges with respect to
$\yometdis{X}$.
In particular,
$\min\{\yometdis{X},1\}$
is a complete metric on
$\yocpm{X}$.
\end{lemma}

\begin{proof}
By \cite[Lemma~2.4]{Ishiki2025SpacesMetricsBaire},
the space
$\yocpm{X}$
is complete with respect to the supremum metric.
The truncated metric
$\min\{\yometdis{X},1\}$
has the same Cauchy
sequences and induces the same topology.
\end{proof}

\begin{definition}\label{def:ultrametric-pseudometric-space}
Let
$X$
be a topological space and let
$R\subseteq [0,\infty)$
be a range set,
namely,
a set containing
$0$.
We call a range set
$R$
\yoemph{characteristic}
 if,
 for every
 $t>0$,
there exists an
element
$r\in R\setminus\{0\}$
such that
$r\le t$.
Equivalently,
$0$
is an accumulation point of
$R\setminus\{0\}$.
A continuous pseudometric
$d$
on
$X$
is an
\yoemph{%
$R$-valued
pseudo-ultrametric}
 if,
for all
$x,y,z\in X$,
we have
\[
d(X^2)\subseteq R
\quad\text{and}\quad
d(x,z)\le d(x,y)\lor d(y,z).
\]
The set of these continuous pseudo-ultrametrics is denoted by
$\youcpm{X}{R}$.
For
$d,e\in\youcpm{X}{R}$,
we denote the set of admissible thresholds for
$d$
and
$e$
by
\begin{equation}\label{eq:ultrametric-admissible-threshold-set}
\youadm_X^R(d,e)
=\left\{\varepsilon\in R\cup\{\infty\}\ \middle|\
\begin{array}{l}
d(x,y)\le e(x,y)\lor\varepsilon,\\
e(x,y)\le d(x,y)\lor\varepsilon
\end{array}
\ \text{for all }x,y\in X\right\}.
\end{equation}
The extended non-Archimedean uniform metric on
$\youcpm{X}{R}$
is defined by
\begin{equation}\label{eq:non-archimedean-uniform-metric}
\youmetdis_X^R(d,e)=\inf\youadm_X^R(d,e).
\end{equation}
Its value may be infinite.
We equip
$\youcpm{X}{R}$
with the topology induced by
$\min\{\youmetdis_X^R,1\}$.

The subspace of all compatible
$R$-valued
ultrametrics on
$X$
is denoted by
$\youltsp{X}{R}$.
We write
$\youcompm{X}{R}$
for the set of complete members of
$\youltsp{X}{R}$.
\end{definition}

\begin{lemma}
\label{lem:ucpm-complete}
Let
$X$
be a topological space and let
$R$
be a characteristic range set.
Then every
sequence in
$\youcpm{X}{R}$
that is Cauchy with respect to
$\youmetdis_X^R$
converges with respect to
$\youmetdis_X^R$.
In particular,
$\min\{\youmetdis_X^R,1\}$
is a complete metric on
$\youcpm{X}{R}$.
\end{lemma}

\begin{proof}
By \cite[Lemma~2.4]{Ishiki2025SpacesMetricsBaire},
the space
$\youcpm{X}{R}$
is complete with respect to the non-Archimedean uniform
metric.
The  metric
$\min\{\youmetdis_X^R,1\}$
has the same
Cauchy sequences and induces the same topology.
\end{proof}

We will also use an extension that preserves the non-Archimedean uniform
distance and can be chosen to preserve completeness.

\begin{theorem}[{\cite[Theorem~4.7 and (N1)]{Ishiki2023Factorization}}]
\label{thm:ultrametric-isometric-extensor}
Let
$X$
be an ultrametrizable space,
let
$A$
be a nonempty closed
subset of
$X$,
and let
$R$
be a characteristic range set.
Then there exists an isometric embedding
\[
E\colon\youltsp{A}{R}\to\youltsp{X}{R}
\]
such that the following statements hold.
\begin{enumerate}[label=\textup{(I\arabic*)},ref=\textup{(I\arabic*)}]
\item\label{item:ultrametric-extensor-extension}
For every
$u\in\youltsp{A}{R}$,
we have
$E(u)|_{A^2}=u$.

\item\label{item:ultrametric-extensor-isometry}
For all
$u,v\in\youltsp{A}{R}$,
we have
\[
\youmetdis_{X}^{R}(E(u),E(v))
=\youmetdis_{A}^{R}(u,v).
\]

\item\label{item:ultrametric-extensor-completeness}
If
$X$
is completely metrizable,
then
$E$
can be chosen so that
\[
E(\youcompm{A}{R})\subseteq\youcompm{X}{R}.
\]
\end{enumerate}
\end{theorem}

\subsection{Components of spaces of metrics}\label{subsec:metric-components}
The extension estimates below compare metrics at finite uniform distance.
We record the corresponding equivalence classes in both the ambient spaces
and their compatible-metric subspaces.

\begin{definition}\label{def:finite-distance-components}
Let
$Z$
be a metrizable space and let
$R$
be a characteristic range set.
On each of
$\yocpm{Z}$
and
$\yometsp{Z}$,
for two members
$d$
and
$e$
of either space,
we define an equivalence relation by
\[
d\sim e
\quad\Longleftrightarrow\quad
\yometdis{Z}(d,e)<\infty.
\]
Its equivalence classes are called
\yoemph{components}.
For
$d\in\yometsp{Z}$,
we denote its component in
$\yometsp{Z}$
by
\[
\yocclq{d}
=\{e\in\yometsp{Z}\mid\yometdis{Z}(d,e)<\infty\},
\]
and we write
$\yoslamet{Z}=\yometsp{Z}/\!\sim$
for the set of these components.

On each of
$\youcpm{Z}{R}$
and
$\youltsp{Z}{R}$,
for two members
$d$
and
$e$
of either space,
we define
\[
d\sim_R e
\quad\Longleftrightarrow\quad
\youmetdis_Z^R(d,e)<\infty.
\]
The equivalence classes are called
\yoemph{$R$-components}.
For
$d\in\youltsp{Z}{R}$,
we denote its component in
$\youltsp{Z}{R}$
by
\[
\youcclq{d}
=\{e\in\youltsp{Z}{R}\mid\youmetdis_Z^R(d,e)<\infty\},
\]
and write
$\youslamet{Z}{R}=\youltsp{Z}{R}/\!\sim_R$
for the set of these components.
{The components of
$\yocpm{Z}$
and
$\yometsp{Z}$
are clopen in those respective spaces.
Similarly,
the
$R$-components of
$\youcpm{Z}{R}$
and
$\youltsp{Z}{R}$
are clopen in those respective spaces.}
\end{definition}

\begin{lemma}
\label{lem:metric-extensor-pullbacks}
Let
$X$
be a metrizable space,
let
$A$
be a closed subset of
$X$,
and let
$\widetilde E\colon\yocpm{A}\to\yocpm{X}$
be an isometric extensor supplied
by \Cref{thm:metric-isometric-extensor}.
Write
$E=\widetilde E|_{\yometsp{A}}$.
Then the following statements hold.
\begin{enumerate}[label=\textup{(P\arabic*)},ref=\textup{(P\arabic*)}]
\item\label{item:metric-extensor-pullback-compatible}
$\widetilde E^{-1}(\yometsp{X})=\yometsp{A}$.

\item\label{item:metric-extensor-pullback-bounded}
$\widetilde E^{-1}(\yometbd{X})=\yometbd{A}$.

\item\label{item:metric-extensor-pullback-component}
For every
$d\in\yometsp{A}$,
we have
$\widetilde E^{-1}(\yocclq{E(d)})=\yocclq{d}$.

\item\label{item:metric-extensor-pullback-complete}
If
$X$
is completely metrizable and the extensor is chosen as in
property~\ref{item:metric-extensor-completeness} of
\Cref{thm:metric-isometric-extensor},
then
$\widetilde E^{-1}(\yocomp{X})=\yocomp{A}$.
\end{enumerate}
\end{lemma}

\begin{proof}
Property~\ref{item:metric-extensor-bounded-closed-image} of
\Cref{thm:metric-isometric-extensor} yields
property~\ref{item:metric-extensor-pullback-compatible}.
For property~\ref{item:metric-extensor-pullback-bounded},
take
$p\in\yometbd{A}$.
Property~\ref{item:metric-extensor-bounded-closed-image} yields
$\widetilde E(p)=E(p)\in\yometbd{X}$.
Hence
$p\in\widetilde E^{-1}(\yometbd{X})$.
Conversely,
take
$p\in\widetilde E^{-1}(\yometbd{X})$.
Then
property~\ref{item:metric-extensor-pullback-compatible} yields
$p\in\yometsp{A}$,
and restriction to
$A^2$
shows that
$p$
is bounded.
Thus
$p\in\yometbd{A}$.
Property~\ref{item:metric-extensor-pullback-component} follows from
property~\ref{item:metric-extensor-pullback-compatible} and the isometry of
$E$.

For property~\ref{item:metric-extensor-pullback-complete},
take
$p\in\yocomp{A}$.
By completeness preservation,
$\widetilde E(p)=E(p)\in\yocomp{X}$.
Hence
$p\in\widetilde E^{-1}(\yocomp{X})$.
Conversely,
take
$p\in\widetilde E^{-1}(\yocomp{X})$.
Then
property~\ref{item:metric-extensor-pullback-compatible} yields
$p\in\yometsp{A}$.
Since
$A$
is closed in
$X$,
it is closed for the metric
$\widetilde E(p)$,
and its restriction
$p$
is complete.
Thus
$p\in\yocomp{A}$.
\end{proof}

\subsection{Borel pullbacks and non-Borel parameter sets}
Continuous inverse images of Borel sets are Borel.
To prove that a space of compatible metrics or one of its components is non-Borel,
we will construct a continuous map whose inverse image of that set is non-Borel.
We recall the topological facts used in this argument and construct dense non-Borel subsets of cardinality
$\aleph_1$
in the unit interval
$[0,1]$
and Cantor space
$\yocantor$.

\begin{lemma}
\label{lem:continuous-pullback-obstruction}
Let
$P$
and
$Z$
be topological spaces,
let
$A\subseteq P$
and
$B\subseteq Z$,
and let
$f\colon P\to Z$
be continuous.
If
$f^{-1}(B)=A$
and
$A$
is non-Borel in
$P$,
then
$B$
is non-Borel in
$Z$.
If,
in addition,
$P$
is compact,
$Z$
is Hausdorff,
and
$f$
is injective,
then
$f(A)$
is a closed subspace of
$B$
homeomorphic to
$A$.
\end{lemma}

\begin{proof}
{For the sake of contradiction,
suppose that
$B$
is Borel in
$Z$.
Then its continuous inverse image
$A$
is Borel in
$P$,
contrary to the assumption.}
For the second assertion,
$f$
is a homeomorphism from
$P$
onto the closed
subspace
$f(P)$
of
$Z$,
and
$f(A)=f(P)\cap B$.
\end{proof}

\begin{lemma}
\label{lem:relative-borel-inheritance}
Let
$Z$
be a topological space and let
$B$
be a Borel subset of
$Z$.
Every subset of
$B$
that is open or closed in
$B$
is Borel in
$Z$.
\end{lemma}

\begin{proof}
Every relatively open or relatively closed subset of
$B$
is the intersection
of
$B$
with an open or closed subset of
$Z$,
respectively.
\end{proof}

\begin{theorem}[
{\cite[Theorem~3.11]{MR1321597}}]
\label{thm:complete-metrizability-g-delta}
Let
$X$
be a metrizable space and let
$A\subseteq X$.
If
$A$
is completely metrizable,
then
$A$
is a
$G_\delta$
subset of
$X$.
Conversely,
if
$X$
is completely metrizable and
$A$
is a
$G_\delta$
subset
of
$X$,
then
$A$
is completely metrizable.
\end{theorem}

We denote the cardinality of the continuum by
$\yoconti$.
A subset
$A$
 of a Polish space is said to have the
 \yoemph{perfect set
property}
if either
 $A$
  is countable or
  $A$
  contains a nonempty perfect
subset.
The perfect set theorem for
Borel sets asserts that every Borel
subset of
a Polish space has the perfect set property
(see
\cite[Theorem~13.6]{MR1321597}).
 Consequently,
 every uncountable Borel
subset of a Polish space has cardinality
$\yoconti$.
The
\yoemph{continuum hypothesis},
abbreviated \yoemph{%
$\yoch$%
},
is the
assertion
$\yoconti=\aleph_{1}$.

Let
$P$
 be a nonempty perfect Polish space.
A subset
$B$
of
$P$
 is called
  \yoemph{Bernstein}
 if,
 for every nonempty
perfect subset
$K$
of
$P$,
both
$K\cap B$
 and
$K\setminus B$
are nonempty
(for the existence,
see \cite[Example 8.24]{MR1321597}).

\begin{proposition}\label{prop:dense-aleph1-nonborel-interval}
There exists a dense subset
$S$
of
$[0,1]$
such that
$\card(S)=\aleph_{1}$
and
$S$
is not
Borel in
$[0,1]$.
\end{proposition}

\begin{proof}
We divide
the proof according to the continuum hypothesis.

\yocase{1}{%
$\yoch$
holds}
Let
$S$
be a Bernstein subset of
$[0,1]$
of cardinality
 $\yoconti$.
Every nonempty relatively open subset of
$[0,1]$
contains a nonempty perfect set,
which meets
$S$
by the definition of a Bernstein set.
Thus
$S$
is dense in
$[0,1]$.
By the definition of a Bernstein set again,
we see that
$S$
contains no nonempty perfect subset.
Since
$S$
is uncountable,
the perfect set theorem for Borel sets shows that
$S$
is non-Borel.
 Since
 $\card([0,1])=\yoconti=\aleph_1$,
 we have
$\card(S)=\aleph_1$.

\yocase{2}{%
$\yoch$
fails}
Fix a countable dense subset
$Q$
of
 $[0,1]$.
Since
 $\aleph_1<\yoconti=\card([0,1])$,
 we may choose a subset
$S$
of
$[0,1]$
 with
 $Q\subseteq S$
 and
  $\card(S)=\aleph_1$.
The set
$S$
is uncountable and has cardinality
$<\yoconti$.
Thus the perfect set theorem for Borel sets implies that
 $S$
 is not Borel.
\end{proof}

We use the notation
$\yocantor$
for the Cantor space,
where
$2=\{0,1\}$
has the discrete topology and
$\yocantor$
has the product topology.

\begin{proposition}\label{prop:dense-aleph1-nonborel-cantor}
There exists a subset
$S$
of
$\yocantor$
such that
$\card(S)=\aleph_{1}$,
$S$
and
$\yocantor\setminus S$
are dense in
$\yocantor$,
and
 $S$
  is non-Borel  in
  $\yocantor$.
\end{proposition}

\begin{proof}
  Take two
  countable dense subsets
  $P$
  and
  $Q$
  of
  $\yocantor$
  with
  $P\cap Q=\emptyset$.
By
\cite[Theorem~3.3.13]{Srivastava1998CourseBorelSets},
there exists a Borel isomorphism
$f\colon[0,1]\to\yocantor$.
Let
$S_0$
 be as in
 \Cref{prop:dense-aleph1-nonborel-interval}.
Put
$S=Q\cup(f(S_0)\setminus P)$.
Then
$S$
is a required set.
\end{proof}

\section{Archimedean Results}\label{sec:archimedean-results}
We first construct a topological embedding of the unit interval
$[0,1]$
into the space
$\yocpm{\yodis}$
of continuous pseudometrics on the discrete space
$\yodis$
of cardinality
$\aleph_1$.
Under this embedding,
the inverse image of the space of compatible metrics is a non-Borel subset of the unit interval
$[0,1]$.
We then use metric extension operators to prove that the space of compatible metrics on any nonseparable metrizable space,
as well as each of its components,
is non-Borel.
We also characterize complete metrizability of spaces of compatible metrics and their components in terms of
$\sigma$-compactness of the underlying space.

\subsection{Non-Borelness on discrete spaces}\label{subsec:metric-discrete-nonborel}
We construct metrics that induce the discrete topology precisely when their parameter lies outside a fixed non-Borel subset of the unit interval
$[0,1]$.

Fix a set
$S\subseteq[0,1]$
obtained from
\Cref{prop:dense-aleph1-nonborel-interval},
and fix a one-to-one enumeration
\begin{equation}\label{eq:metric-coding-parameter-set}
S=\{w_\alpha\mid\alpha<\omega_1\}.
\end{equation}
Put
\begin{equation}\label{eq:metric-coding-underlying-set}
\yodis=\{p,q\}\sqcup\{b_\alpha\mid\alpha<\omega_1\}
\sqcup\{c_{\alpha,n}\mid\alpha<\omega_1,\ n\in\Z_{\ge 0}\}
\end{equation}
For
$\alpha<\omega_1$,
put
\begin{equation}\label{eq:metric-coding-cluster}
C_\alpha=\{b_\alpha\}\sqcup\{c_{\alpha,n}\mid n\in\Z_{\ge 0}\}.
\end{equation}
Then
$\card(\yodis)=\aleph_{1}$.
We equip
$\yodis$
with the discrete topology.

\begin{lemma}\label{lem:metric-coding}
For
$t\in[0,1]$,
$\alpha<\omega_1$,
and
 $n\in\Z_{\ge 0}$,
put
\begin{equation}\label{eq:metric-coding-length}
l_n(\alpha,t)=\abs{t-w_\alpha}+2^{-n}.
\end{equation}
{Fix
$t\in[0,1]$.
Set
$d_t(u,u)=0$
for every
$u\in\yodis$.
For distinct
$u,v\in\yodis$,
define
$d_t(u,v)$
symmetrically by the following formula with indices
$\alpha<\omega_1$
and
$m,n\in\Z_{\ge 0}$
}
\begin{equation}\label{eq:metric-coding-definition}
d_t(u,v)=
\begin{cases}
1+t &\{u,v\}=\{p,q\},\\
l_n(\alpha,t) &\{u,v\}=\{b_\alpha,c_{\alpha,n}\},\\
l_n(\alpha,t)+l_m(\alpha,t)
  &\{u,v\}=\{c_{\alpha,n},c_{\alpha,m}\},\ n\ne m,\\
4 &\text{otherwise}.
\end{cases}
\end{equation}
Then
$d_t$
is a bounded metric on the underlying set of
$\yodis$,
and the following statements hold.
\begin{enumerate}[label=\textup{(A\arabic*)},ref=\textup{(A\arabic*)}]
\item\label{item:metric-coding-completeness}
We have
$d_t\in\yocomp{\yodis}$
exactly for
$t\in[0,1]\setminus S$.
\item\label{item:metric-coding-compatibility}
We have
$d_t\in\yometsp{\yodis}$
exactly for
$t\in[0,1]\setminus S$.
\end{enumerate}
\end{lemma}

\begin{proof}
Starting with
$\{b_\alpha\}$,
we inductively apply one-point extensions using
\Cref{lem:one-point-amalgamation},
adjoining
$c_{\alpha,n}$
at distance
$l_n(\alpha,t)$
from
$b_\alpha$
for
$n=0,1,\ldots$.
The union of these extensions is
$d_t|_{C_\alpha^2}$.
Distinct members of
$\bigl\{\{p,q\}\bigr\}\cup\{C_\alpha\mid\alpha<\omega_1\}$
have distance
$4$,
while every member has diameter at most
 $4$.
Thus
$d_t$
is a metric and
$d_t(\yodis^2)\subseteq[0,4]$.

We now prove assertions
\ref{item:metric-coding-completeness} and
\ref{item:metric-coding-compatibility}.
We first prove assertion~\ref{item:metric-coding-completeness}.
For this purpose,
we first determine when
$d_t$
induces the discrete topology.
Assume that
$t\in [0, 1]\setminus S$.
Then
for every
$\alpha<\omega_1$,
and
for
every
$n\in \Z_{\ge 0}$,
we have
$\inf_{v\ne c_{\alpha,n}}d_t(c_{\alpha,n},v)
=l_n(\alpha,t)\ge2^{-n}$
and
$\inf_{v\ne b_\alpha}d_t(b_\alpha,v)
=\abs{t-w_\alpha}>0$.
Similarly,
we have
$d_t(p,q)=1+t>0$.
Thus,
$d_t$
induces
the discrete topology.

Conversely,
assume that
$t=w_\beta\in S$
for some
$\beta<\omega_1$.
Then we have
$d_t(b_\beta,c_{\beta,n})=2^{-n}\to0$,
as
$n\to \infty$.
Hence
the singleton
$\{b_\beta\}$
is not open for the topology induced by
$d_t$.
This proves assertion~\ref{item:metric-coding-compatibility}.

It remains to verify completeness for
assertion~\ref{item:metric-coding-completeness}.
If
$t\in S$,
then the preceding argument shows that
$d_t\notin\yometsp{\yodis}$,
and hence
$d_t\notin\yocomp{\yodis}$.
Assume that
$t\in[0,1]\setminus S$.
A Cauchy sequence is eventually contained in one of the sets
$\{p,q\}$
or
$C_\alpha$,
since distinct such sets have distance
$4$.
Within
$C_\alpha$,
distinct points have distance at least
$\abs{t-w_\alpha}>0$,
so the sequence is eventually constant.
Similarly,
every Cauchy sequence in
$\{p,q\}$
is eventually constant.
Thus every Cauchy sequence is eventually constant,
and
$d_t$
is complete.
This proves assertion~\ref{item:metric-coding-completeness}.
\end{proof}

\begin{lemma}\label{lem:closed-parameter-embedding}
The map
$\Phi\colon[0,1]\to\yocpm{\yodis}$
defined by
$\Phi(t)=d_t$
is a bi-Lipschitz embedding,
and
the image
$\Phi([0, 1]\setminus S)$
is closed in
$\yometsp{\yodis}$.
\end{lemma}

\begin{proof}
Fix
$s,t\in[0,1]$.
Equation \eqref{eq:metric-coding-definition} shows that
$d_s(p,q)=1+s$
and
$d_t(p,q)=1+t$.
By the definition of
$\yometdis{\yodis}$,
we therefore obtain
\[
\begin{aligned}
\yometdis{\yodis}(d_s,d_t)
&=\sup_{(u,v)\in\yodis^2}\abs{d_s(u,v)-d_t(u,v)}
\ge\abs{d_s(p,q)-d_t(p,q)}\\
&=\abs{(1+s)-(1+t)}=\abs{s-t}.
\end{aligned}
\]
For every
$\alpha<\omega_1$
and
$n\in\Z_{\ge 0}$,
\eqref{eq:metric-coding-length} implies
$\abs{l_n(\alpha,s)-l_n(\alpha,t)}\le\abs{s-t}$.
Applying \eqref{eq:metric-coding-definition} to each pair of points,
we obtain
\begin{equation}\label{eq:metric-coding-uniform-estimate}
\abs{s-t}\le\yometdis{\yodis}(d_s,d_t)\le2\abs{s-t}.
\end{equation}
Thus
$\Phi$
is a bi-Lipschitz embedding.
Property~\ref{item:metric-coding-compatibility} of
\Cref{lem:metric-coding} yields
\begin{equation}\label{eq:metric-parameter-compatible-pullback}
\Phi^{-1}(\yometsp{\yodis})=[0,1]\setminus S.
\end{equation}
Since
$[0,1]$
is compact and
$\yocpm{\yodis}$
is Hausdorff,
the second
assertion of \Cref{lem:continuous-pullback-obstruction} shows that
$\Phi([0,1]\setminus S)$
is closed in
$\yometsp{\yodis}$.
\end{proof}

\begin{theorem}\label{thm:bounded-discrete-nonborel}
    Let
    $\yodis$
    be a
    discrete space of
    cardinality
    $\aleph_{1}$.
    Then
    $\yometbd{\yodis}$
    is non-Borel in
    $\yocpm{\yodis}$.
\end{theorem}
\begin{proof}
Let
$\Phi\colon[0,1]\to\yocpm{\yodis}$
be a bi-Lipschitz embedding constructed in
\Cref{lem:closed-parameter-embedding}.
All members of the family
$\{d_t\}_{t\in[0,1]}$
are bounded.
Thus \eqref{eq:metric-parameter-compatible-pullback} yields
$\Phi^{-1}(\yometbd{\yodis})=[0,1]\setminus S$.
The conclusion follows from
\Cref{lem:continuous-pullback-obstruction,prop:dense-aleph1-nonborel-interval}.
\end{proof}

\begin{theorem}\label{thm:metric-component-nonborel}
Let
$\yodis$
be a discrete space of cardinality
$\ge \aleph_1$.
Then every component
$\yoccl\in\yoslamet{\yodis}$
is non-Borel in
$\yocpm{\yodis}$.
\end{theorem}

\begin{proof}
Fix
$d\in\yoccl$
 and a point
 $z\in \yodis$.
 Since
\[
\yodis=\bigcup_{k\in\Z_{\ge0}}\yoclosedball{z}{2^k}{d},
\]
there exists
$k\in\Z_{\ge 0}$
 such that
$L=\yoclosedball{z}{2^k}{d}$
has cardinality
$\ge \aleph_1$.
Choose
a subset
$H$
of
$L$
with
$z\in H$
and
$\card(H)=\aleph_1$.
Let
$\yocodesp$
be a copy of the discrete space used in
\Cref{lem:metric-coding}.
Write
$\{d_t\}_{t\in[0,1]}$
for the corresponding family of metrics on
$\yocodesp$,
and
$p,q\in\yocodesp$
for the points corresponding to
$p$
and
$q$
in that lemma.
Choose
a bijection
$\theta\colon \yocodesp\to H$
with
$\theta(p)=z$.
For
$t\in[0,1]$
and
$a,b\in H$,
define
\begin{equation}\label{eq:metric-transferred-coding}
e_t(a,b)=d_t\bigl(\theta^{-1}(a),\theta^{-1}(b)\bigr).
\end{equation}
Thus we see that
$e_t(H^2)\subseteq[0,4]$,
and
the transferred form of property
\ref{item:metric-coding-compatibility}
of \Cref{lem:metric-coding} yields
$e_t\in\yometsp{H}$
if and only if
$t\in[0,1]\setminus S$.

For
$t\in[0,1]$,
define a metric
$m_t$
on
$\yodis$
as
follows.
Its restrictions to
$H$
and
$(\yodis\setminus H)\cup\{z\}$
are
$e_t$
and
$d$,
respectively,
and for
$x\in \yodis\setminus H$
and
$a\in H$
put
\begin{equation}\label{eq:metric-amalgamation-cross-pair}
m_t(x,a)=m_t(a,x)=d(x,z)+e_t(z,a).
\end{equation}
This is the amalgamation of the two metric spaces
along their common
point
$z$.
By \Cref{lem:one-point-amalgamation},
$m_t$
is a metric.

For each
$x\in\yodis\setminus H$
and
$a\in H$,
we have
$m_t(x,a)\ge d(x,z)>0$.
Since
$d$
induces the discrete topology,
$\inf_{x\in\yodis\setminus\{z\}}d(x,z)>0$.
Thus the amalgamation is compatible precisely when
$e_t$
is compatible on
$H$.
Consequently,
\begin{equation}\label{eq:metric-amalgamation-compatibility}
m_t\in\yometsp{\yodis}
\quad\Longleftrightarrow\quad
t\in[0,1]\setminus S.
\end{equation}
Moreover,
on
$H^2$
the metrics
$d$
and
$m_t$
are bounded by
$2^{k+1}$
and
$4$,
respectively.
On
$((\yodis\setminus H)\cup\{z\})^2$
they agree.
For
$x\in\yodis\setminus H$
and
$a\in H$,
the triangle inequality yields
\[
\abs{m_t(x,a)-d(x,a)}\le d(z,a)+e_t(z,a)\le2^k+4.
\]
Hence
\begin{equation}\label{eq:metric-amalgamation-component-bound}
\yometdis{\yodis}(m_t,d)<\infty.
\end{equation}
Consequently,
by \eqref{eq:metric-amalgamation-compatibility} and
\eqref{eq:metric-amalgamation-component-bound},
we have
$m_t\in\yoccl$
whenever
$t\in[0,1]\setminus S$.

Define
a map
$\Psi\colon[0,1] \to\yocpm{\yodis}$
by
$\Psi(t)=m_t$.
Let
 $q_H=\theta(q)$.
Fix
$s,t\in[0,1]$,
and recall that
$\theta(p)=z$
and
$\theta(q)=q_H$.
By \eqref{eq:metric-transferred-coding} and
\eqref{eq:metric-coding-definition},
we have
\[
\begin{aligned}
\yometdis{\yodis}(m_s,m_t)
&\ge\abs{m_s(z,q_H)-m_t(z,q_H)}
=\abs{d_s(p,q)-d_t(p,q)}\\
&=\abs{(1+s)-(1+t)}
=\abs{s-t}.
\end{aligned}
\]
For the upper bound,
let
$a,b\in H$.
Equations \eqref{eq:metric-transferred-coding} and
\eqref{eq:metric-coding-uniform-estimate} imply
\[
\begin{aligned}
\abs{m_s(a,b)-m_t(a,b)}
&=\abs{e_s(a,b)-e_t(a,b)}\\
&\le\yometdis{\yocodesp}(d_s,d_t)
\le2\abs{s-t}.
\end{aligned}
\]
On
$((\yodis\setminus H)\cup\{z\})^2$,
we have
$m_s=m_t=d$.
For
$x\in\yodis\setminus H$
and
$a\in H$,
the term
$d(x,z)$
cancels in \eqref{eq:metric-amalgamation-cross-pair},
so
\[
\begin{aligned}
\abs{m_s(x,a)-m_t(x,a)}
&=\abs{e_s(z,a)-e_t(z,a)}\\
&\le2\abs{s-t}.
\end{aligned}
\]
By symmetry,
these cases cover all pairs of points in
$\yodis$.
Taking the supremum over
$\yodis^2$,
we obtain
\begin{equation}\label{eq:metric-amalgamation-bilipschitz}
\abs{s-t}
\le\yometdis{\yodis}(m_s,m_t)
\le 2\abs{s-t}.
\end{equation}
Equation \eqref{eq:metric-amalgamation-bilipschitz} shows that
$\Psi$
is a bi-Lipschitz embedding.
Equations \eqref{eq:metric-amalgamation-compatibility} and
\eqref{eq:metric-amalgamation-component-bound} yield
\begin{equation}\label{eq:metric-component-coding-pullback}
\Psi^{-1}(\yoccl)=[0,1]\setminus S.
\end{equation}
Thus the conclusion follows from
\Cref{lem:continuous-pullback-obstruction,prop:dense-aleph1-nonborel-interval}.
\end{proof}

\begin{corollary}\label{cor:discrete-metric-nonborel}
  Let
$\yodis$
be a discrete space of cardinality
$\aleph_1$.
Then
$\yometsp{\yodis}$
is non-Borel in
$\yocpm{\yodis}$.
\end{corollary}

\begin{proof}
Every component is open in
$\yometsp{\yodis}$.
Thus the conclusion follows from
\Cref{thm:metric-component-nonborel,lem:relative-borel-inheritance}.
\end{proof}

\subsection{Borelness and separability}\label{subsec:metric-borel-separability}
We now extend the results for discrete spaces to general metrizable spaces by applying metric extension operators to closed discrete subspaces.

\begin{theorem}
\label{thm:borel-separability}
Let
$X$
be a metrizable space.
If
$\yometbd{X}$
 is Borel  in
$\yocpm{X}$,
then
$X$
is separable.
\end{theorem}

\begin{proof}
For the sake of contradiction,
suppose that
$X$
is nonseparable.
By \Cref{lem:maximal-separated-set},
take a closed discrete subspace
$\yodis$
of
$X$
with
$\card(\yodis)=\aleph_1$.
Using property~\ref{item:metric-extensor-bounded-closed-image} of
\Cref{thm:metric-isometric-extensor},
take  an isometric
extensor
$\widetilde{E}\colon\yocpm{\yodis}\to \yocpm{X}$
whose restriction
$E=\widetilde E|_{\yometsp{\yodis}}$
maps
$\yometsp{\yodis}$
into
$\yometsp{X}$.
Property~\ref{item:metric-extensor-pullback-bounded} of
\Cref{lem:metric-extensor-pullbacks} yields
$\widetilde E^{-1}(\yometbd{X})=\yometbd{\yodis}$.
Since
$\yometbd{X}$
is Borel by hypothesis,
its continuous inverse image
$\yometbd{\yodis}$
would be Borel.
This contradicts
\Cref{thm:bounded-discrete-nonborel}.
\end{proof}

For a metrizable space
$X$,
we now consider each component of
$\yometsp{X}$,
using the argument for bounded compatible metrics as a guide.
The proof of the following theorem uses the explicit construction of the
isometric extensor in \cite{Ishiki2024IsometricExtensor}.
To obtain an extension in a prescribed component of
$\yometsp{X}$,
we choose the base metrics in that construction and use the estimate in
\cite[Lemma~4.4]{Ishiki2024IsometricExtensor}.

\begin{theorem}\label{cor:nonseparable-metric-nonborel}
  Let
  $X$
  be a nonseparable metrizable space.
  Then the following statements are true.
  \begin{enumerate}[label=\textup{(M\arabic*)},ref=\textup{(M\arabic*)}]
    \item\label{item:nonseparable-metric-nonborel-components}
    Each component of
    $\yometsp{X}$
    is non-Borel in
    $\yocpm{X}$.
    \item\label{item:nonseparable-metric-nonborel-space}
    $\yometsp{X}$
    is non-Borel in
    $\yocpm{X}$.
\end{enumerate}
\end{theorem}

\begin{proof}
We first prove assertion~\ref{item:nonseparable-metric-nonborel-components}.
Let
$\yoccl_X\in\yoslamet{X}$
be arbitrary,
and choose
$d\in\yoccl_X$.
By \Cref{lem:maximal-separated-set},
choose an
$\varepsilon>0$
for which
$X$
has
no countable
$\varepsilon$-net,
and let
$\yodis$
be a maximal
$\varepsilon$-separated
subset of
$(X,d)$.
Then
$\yodis$
is an uncountable closed discrete
$\varepsilon$-net
in
$(X,d)$.
Put
\begin{equation}\label{eq:separated-net-restriction}
e=d|_{\yodis^2}
\end{equation}
and let
$\yoccl_{\yodis}$
be the component of
$e$
in
$\yometsp{\yodis}$.
\Cref{thm:metric-component-nonborel}
shows that
$\yoccl_{\yodis}$
is non-Borel in
$\yocpm{\yodis}$.
Apply the construction in the proof of
\cite[Theorem~1.1]{Ishiki2024IsometricExtensor}
to the closed subset
$\yodis$
and choose
$e$
and its extension
$d$
as the base metrics for the component
$\yoccl_{\yodis}$.
This produces an isometric extensor as in
\Cref{thm:metric-isometric-extensor},
\[
\widetilde E\colon
\yocpm{\yodis}\to\yocpm{X}.
\]
Write
$E=\widetilde E|_{\yometsp{\yodis}}$
and
$r=E(e)$.
For the proof of \Cref{cor:nonseparable-metric-nonborel},
we need a pointwise upper bound for the extended metric in terms of the original metric.
Namely,
the estimate in
\cite[Lemma~4.4]{Ishiki2024IsometricExtensor}
shows that for every
$x\in X$
and every
$a\in\yodis$,
\begin{equation}\label{eq:metric-extensor-point-estimate}
r(x,a)\le 4d(x,a)+1.
\end{equation}
For each
$x\in X$,
choose
$a_x\in\yodis$
with
$d(x,a_x)<\varepsilon$.
For
$x,y\in X$,
the triangle inequalities and
$r|_{\yodis^2}=d|_{\yodis^2}$
yield
\[
\begin{aligned}
\abs{r(x,y)-d(x,y)}
&\le r(x,a_x)+r(y,a_y)+d(x,a_x)+d(y,a_y)\\
&\le 5d(x,a_x)+5d(y,a_y)+2,
\end{aligned}
\]
where the second inequality follows from
\eqref{eq:metric-extensor-point-estimate}.
Therefore,
for all
$x,y\in X$,
\begin{equation}\label{eq:metric-extensor-component-bound}
\abs{r(x,y)-d(x,y)}<10\varepsilon+2.
\end{equation}
Thus \eqref{eq:metric-extensor-component-bound} implies that
$r\in\yoccl_X$.
Since
$\yoccl_X=\yocclq{r}$,
property~\ref{item:metric-extensor-pullback-component} of
\Cref{lem:metric-extensor-pullbacks} yields
$\widetilde E^{-1}(\yoccl_X)=\yoccl_{\yodis}$.
The non-Borelness of
$\yoccl_{\yodis}$
and
\Cref{lem:continuous-pullback-obstruction} show that
$\yoccl_X$
is
non-Borel in
$\yocpm{X}$.
Hence every component of
$\yometsp{X}$
is non-Borel in
$\yocpm{X}$.

Every component is open in
$\yometsp{X}$,
so
assertion~\ref{item:nonseparable-metric-nonborel-space}
follows from
assertion~\ref{item:nonseparable-metric-nonborel-components} and
\Cref{lem:relative-borel-inheritance}.
\end{proof}

\subsection{Complete metrizability and
\texorpdfstring{$\sigma$}{sigma}-compactness}\label{subsec:metric-complete-metrizability}
We combine the preceding Borelness results with Koshino's theorem to characterize complete metrizability of spaces of compatible metrics and their components in terms of
$\sigma$-compactness of the underlying space.

\begin{corollary}\label{cor:metric-space-complete-separability}
Let
$X$
be a metrizable space.
If
$\yometsp{X}$
is completely
metrizable,
then
$X$
is separable.
\end{corollary}

\begin{proof}
Now assume that
 $\yometsp{X}$
  is completely metrizable.
  By
\Cref{thm:complete-metrizability-g-delta},
it is a
$G_\delta$
subspace of
$\yocpm{X}$.
Consequently,
$\yometsp{X}$
is Borel in
$\yocpm{X}$.
By the contrapositive of
\Cref{cor:nonseparable-metric-nonborel},
$X$
is separable.
\end{proof}

Combining
\Cref{thm:koshino-characterization,thm:complete-metrizability-g-delta,thm:borel-separability}
with the arguments for components below,
we obtain the following equivalences.

\begin{corollary}\label{cor:bounded-metric-complete-sigma-compact}
Let
$X$
be a metrizable space.
Then
the following statements are equivalent.
\begin{enumerate}[label=\textup{(C\arabic*)},ref=\textup{(C\arabic*)}]
  \item\label{cor:bounded-metric-complete-sigma-compact:met}
  $\yometsp{X}$
  is completely metrizable.
  \item\label{cor:bounded-metric-complete-sigma-compact:bmet}
  $\yometbd{X}$
  is completely metrizable.
  \item\label{cor:bounded-metric-complete-sigma-compact:each-component}
  Each component of
$\yometsp{X}$
is completely metrizable.
  \item\label{cor:bounded-metric-complete-sigma-compact:some-component}
  Some component of
$\yometsp{X}$
is completely metrizable.
  \item\label{cor:bounded-metric-complete-sigma-compact:sigma-compact}
  $X$
  is
  $\sigma$-compact.
\end{enumerate}
\end{corollary}

\begin{proof}

If
$X=\emptyset$,
then the space of metrics consists of the unique empty metric,
and all five assertions hold.
Henceforth assume that
$X\ne\emptyset$.

\medskip
\noindent\textbf{\ref{cor:bounded-metric-complete-sigma-compact:met}
$\Longrightarrow$
\ref{cor:bounded-metric-complete-sigma-compact:each-component}:}
Every component is a clopen subspace of
$\yometsp{X}$,
and every closed subspace of a completely metrizable space is completely
metrizable.

\medskip
\noindent\textbf{\ref{cor:bounded-metric-complete-sigma-compact:each-component}
$\Longrightarrow$
\ref{cor:bounded-metric-complete-sigma-compact:met}:}
$\yometsp{X}$
is the topological sum of its clopen components.
A topological sum of completely metrizable spaces is completely metrizable.

\medskip
\noindent\textbf{\ref{cor:bounded-metric-complete-sigma-compact:each-component}
$\Longrightarrow$
\ref{cor:bounded-metric-complete-sigma-compact:bmet}:}
The bounded metrics form one component of
$\yometsp{X}$,
so they are completely metrizable.

  \medskip
\noindent\textbf{\ref{cor:bounded-metric-complete-sigma-compact:bmet}
$\Longrightarrow$
\ref{cor:bounded-metric-complete-sigma-compact:sigma-compact}:}
Assume that
$\yometbd{X}$
is completely metrizable.
By \Cref{thm:complete-metrizability-g-delta},
it is a
$G_\delta$
subspace of
$\yocpm{X}$.
Thus
\Cref{thm:borel-separability}
implies that
$X$
 is separable.
Therefore
by \Cref{thm:koshino-characterization},
$X$
is
$\sigma$-compact.

\medskip
\noindent\textbf{\ref{cor:bounded-metric-complete-sigma-compact:sigma-compact}
$\Longrightarrow$
\ref{cor:bounded-metric-complete-sigma-compact:each-component}:}
Assume that
$X$
is
$\sigma$-compact.
Let
$\yoccl$
be a component of
$\yometsp{X}$,
and choose
$d\in\yoccl$.
Put
\[
\yopd{d}=\{p\in\yocpm{X}\mid\yometdis{X}(p,d)<\infty\}.
\]
The space
$\yopd{d}$
is complete with respect to
$\yometdis{X}$.
Fix
$o\in X$.
For
$n\in\Z_{\ge 0}$,
put
\begin{equation}\label{eq:component-exhaustion-balls}
A_n=\yoclosedball{o}{2^n}{d}.
\end{equation}

Each
$A_n$
is
$\sigma$-compact,
and hence separable.
By \Cref{thm:koshino-characterization},
$\yometbd{A_n}$
is completely
metrizable,
so \Cref{thm:complete-metrizability-g-delta} shows that it is a
$G_\delta$
subspace of
$\yocpmbd{A_n}$.
{For each
$n\in\Z_{\ge 0}$,
the restriction map
}\begin{equation}\label{eq:component-restriction-map}
r_n\colon \yopd{d}\to\yocpmbd{A_n},
\qquad
r_n(p)=p|_{A_n^2},
\end{equation}
{is continuous.
We will show that}
\begin{equation}\label{eq:component-restriction-characterization}
\yoccl=\bigcap_{n\in\Z_{\ge 0}}r_n^{-1}(\yometbd{A_n}).
\end{equation}
Take
$p\in\yoccl$.
For every
$n\in\Z_{\ge 0}$,
the restriction
$r_n(p)$
is compatible on
$A_n$
and is bounded because
\[
\diam(A_n,r_n(p))\le 2^{n+1}+\yometdis{X}(p,d)<\infty.
\]
Thus
\[
p\in\bigcap_{n\in\Z_{\ge 0}}r_n^{-1}(\yometbd{A_n}).
\]
Conversely,
take
\[
p\in\bigcap_{n\in\Z_{\ge 0}}r_n^{-1}(\yometbd{A_n}).
\]
Put
$K=\yometdis{X}(p,d)$.
Every distinct pair of points belongs to some
$A_n$,
so
$p$
is a metric.
It remains to verify
that
$p$
generates the topology of
$X$.
For a point
$x\in X$
and a neighborhood
$U$
of
$x$,
choose
$N\in \zz_{\ge 0}$
such that
$2^N>d(o,x)+K+1$
and
$x\in A_N$.
If
$y\in X\setminus A_N$,
then
the triangle inequality and the choice of
$N$
yield
\begin{equation}\label{eq:component-localization-d-bound}
d(x,y)
\ge d(o,y)-d(o,x)
>2^N-d(o,x)
>K+1.
\end{equation}
Since
$\yometdis{X}(p,d)=K$,
{for every
$y\in X\setminus A_N$,}
we obtain
\begin{equation}\label{eq:component-localization-p-bound}
p(x,y)\ge d(x,y)-K>1.
\end{equation}

Since
$p|_{A_N^2}$
is compatible with the topology of
$A_N$,
the set
$U\cap A_N$
is a neighborhood of
$x$
in
$A_N$.
Hence there is an
$\varepsilon\in(0,1)$
such that
\begin{equation}\label{eq:component-restricted-neighborhood}
\yoopenball{x}{\varepsilon}{p|_{A_N^2}}\subseteq U\cap A_N.
\end{equation}
Take
$y\in\yoopenball{x}{\varepsilon}{p}$.
Since
$p(x,y)<\varepsilon<1$,
\eqref{eq:component-localization-p-bound} implies that
$y\in A_N$.
Together with \eqref{eq:component-restricted-neighborhood},
this proves
\begin{equation}\label{eq:component-global-neighborhood}
\yoopenball{x}{\varepsilon}{p}
\subseteq
\yoopenball{x}{\varepsilon}{p|_{A_N^2}}
\subseteq U\cap A_N\subseteq U.
\end{equation}
Thus every neighborhood of
$x$
in the original topology contains
$\yoopenball{x}{\varepsilon}{p}$.
Conversely,
for every
$y\in X$
and
$r>0$,
the ball
$\yoopenball{y}{r}{p}$
is open in the original topology because
$p\in\yopd{d}\subseteq\yocpm{X}$.
Hence
$p\in\yometsp{X}$,
and
$\yometdis{X}(p,d)<\infty$
implies that
$p\in\yoccl$.
This proves \eqref{eq:component-restriction-characterization}.
Since each
$r_n^{-1}(\yometbd{A_n})$
is
$G_\delta$,
\eqref{eq:component-restriction-characterization} and
\Cref{thm:complete-metrizability-g-delta} show that
$\yoccl$
is completely
metrizable.

\medskip
\noindent\textbf{\ref{cor:bounded-metric-complete-sigma-compact:each-component}
$\Longrightarrow$
\ref{cor:bounded-metric-complete-sigma-compact:some-component}:}
This is immediate.

\medskip
\noindent\textbf{\ref{cor:bounded-metric-complete-sigma-compact:some-component}
$\Longrightarrow$
\ref{cor:bounded-metric-complete-sigma-compact:sigma-compact}:}
Let
$\yoccl$
be a completely metrizable component of
$\yometsp{X}$.
Choose
$d\in\yoccl$
and
$o\in X$.
For
$n\in\Z_{\ge 0}$,
put
\begin{equation}\label{eq:component-quotient-balls}
A_n=\yoclosedball{o}{2^n}{d}.
\end{equation}
{Fix
$n\in\Z_{\ge 0}$.
We shall show that the restriction map
}
\begin{equation}\label{eq:component-quotient-restriction}
R_n\colon\yoccl\to\yometbd{A_n},
\qquad R_n(m)=m|_{A_n^2},
\end{equation}
is an open continuous surjection.
For
$m_1,m_2\in\yoccl$,
we have
\begin{equation}\label{eq:component-restriction-lipschitz}
\yometdis{A_n}(R_n(m_1),R_n(m_2))
\le\yometdis{X}(m_1,m_2).
\end{equation}
Thus
$R_n$
is continuous.
We next prove surjectivity and openness.
Take
$m\in\yoccl$
and
$e\in\yometbd{A_n}$.
The interpolation theorem
\cite[Theorem~1.1]{zbMATH08213674}
yields an extension
$M\in\yometsp{X}$
such that
\begin{equation}\label{eq:metric-interpolation-extension}
M|_{A_n^2}=e
\end{equation}
and
\begin{equation}\label{eq:metric-interpolation-distance}
\yometdis{X}(M,m)
=\yometdis{A_n}(e,m|_{A_n^2}).
\end{equation}
Since
$e$
and
$m|_{A_n^2}$
are bounded,
\eqref{eq:metric-interpolation-distance} implies that
$M\in\yoccl$.
Equation \eqref{eq:metric-interpolation-extension} therefore proves that
$R_n$
is surjective.
Combining \eqref{eq:component-restriction-lipschitz} with
\eqref{eq:metric-interpolation-distance},
we obtain,
for every
$m\in\yoccl$
and
$r>0$,
\begin{equation}\label{eq:restriction-map-ball-image}
R_n\!\left(\yoccl\cap\yoopenball{m}{r}{\yometdis{X}}\right)
=\yometbd{A_n}\cap\yoopenball{m|_{A_n^2}}{r}{\yometdis{A_n}}.
\end{equation}
Thus \eqref{eq:restriction-map-ball-image} shows that
$R_n$
is open.
A metrizable continuous open image of a completely metrizable space is
completely metrizable
\cite[Section~3.4]{MichaelStone1969Quotients}.
Applying this result to
$R_n$,
we conclude that
$\yometbd{A_n}$
is completely metrizable.
Applying the implication
\ref{cor:bounded-metric-complete-sigma-compact:bmet}
$\Longrightarrow$
\ref{cor:bounded-metric-complete-sigma-compact:sigma-compact}
to
$A_n$,
we conclude that every
$A_n$
is
$\sigma$-compact,
  and hence so is
  $X$.
\end{proof}

\subsection{Borelness of spaces of complete metrics}\label{subsec:complete-metric-nonborel}
We next consider the Borelness of the space of complete metrics.

\begin{theorem}\label{thm:complete-metric-nonborel}
Let
$\yodis$
be a discrete space of cardinality
$\aleph_1$.
Then for every component
$\yoccl\in\yoslamet{\yodis}$,
the space
$\yoccl\cap\yocomp{\yodis}$
is non-Borel in
$\yocpm{\yodis}$.
  Consequently,
  the set
$\yocomp{\yodis}$
 is also non-Borel
 in
 $\yocpm{\yodis}$.
\end{theorem}

\begin{proof}
Fix
$d\in\yoccl$.
For
$x,y\in\yodis$,
define
\[
\yobard(x,y)=
\begin{cases}
0 & \text{if }x=y,\\
d(x,y)\lor1 & \text{if }x\ne y.
\end{cases}
\]
Then
we have
$\yometdis{\yodis}(d,\yobard)\le 1$,
 and hence
$\yobard\in\yoccl$.
The metric
$\yobard$
 is uniformly discrete and
complete.
{Repeat the construction in
\Cref{thm:metric-component-nonborel}
with
$\yobard$
in place of
$d$,
using \eqref{eq:metric-transferred-coding} to define the metrics
$e_t$
on
$H$
and \eqref{eq:metric-amalgamation-cross-pair} to define
$m_t$
on
$\yodis$.
Fix
$t\in[0,1]\setminus S$.
By \eqref{eq:metric-transferred-coding},
the bijection
$\theta$
is an isometry from
$(\yocodesp,d_t)$
to
$(H,e_t)$.
Thus
$e_t$
is complete by property~\ref{item:metric-coding-completeness} of
\Cref{lem:metric-coding}.
}
The restriction of
$\yobard$
to
$(\yodis\setminus H)\cup\{z\}$
is complete because
$\yobard$
is uniformly
discrete.
By \Cref{lem:one-point-amalgamation},
$m_t$
is complete.
Equation \eqref{eq:metric-amalgamation-component-bound},
applied with
$\yobard$
in place of
$d$,
yields
$\yometdis{\yodis}(m_t,\yobard)<\infty$,
and hence
 $m_t\in\yoccl$.
If
$\Psi(t)=m_t$
is the map constructed in
\Cref{thm:metric-component-nonborel},
then
\begin{equation}\label{eq:complete-metric-coding-pullback}
\Psi^{-1}(\yoccl\cap\yocomp{\yodis})=[0,1]\setminus S.
\end{equation}
Thus \eqref{eq:complete-metric-coding-pullback} shows that
$\yoccl\cap\yocomp{\yodis}$
is non-Borel in
$\yocpm{\yodis}$
by
\Cref{lem:continuous-pullback-obstruction}.

The component of
$\yocpm{\yodis}$
containing
$d$
is clopen
in
$\yocpm{\yodis}$,
and its intersection with
$\yocomp{\yodis}$
equals
$\yoccl\cap\yocomp{\yodis}$.
Hence
\Cref{lem:relative-borel-inheritance} shows that
$\yocomp{\yodis}$
is also
non-Borel in
$\yocpm{\yodis}$.
\end{proof}

\begin{corollary}\label{cor:complete-metric-borel-separability}
Let
$X$
be a completely  metrizable space.
If
$\yocomp{X}$
is Borel
  in
   $\yocpm{X}$,
   then
   $X$
   is
separable.
\end{corollary}

\begin{proof}
For the sake of contradiction,
suppose that
$X$
is nonseparable.
By \Cref{lem:maximal-separated-set},
take a closed discrete subspace
$\yodis$
of cardinality
$\aleph_1$.
Choose,
using the complete metrizability of
$X$,
an isometric embedding
$\widetilde E\colon\yocpm{\yodis}\to\yocpm{X}$
as in
property~\ref{item:metric-extensor-completeness} of
\Cref{thm:metric-isometric-extensor} whose restriction
$E=\widetilde E|_{\yometsp{\yodis}}$
preserves completeness.
Property~\ref{item:metric-extensor-pullback-complete} of
\Cref{lem:metric-extensor-pullbacks} yields
\[
\widetilde E^{-1}(\yocomp{X})=\yocomp{\yodis}.
\]
Since
$\yocomp{X}$
is Borel by hypothesis,
this contradicts
\Cref{lem:continuous-pullback-obstruction,thm:complete-metric-nonborel}.
\end{proof}
\section{Non-Archimedean Results}\label{sec:non-archimedean-results}
We now replace the unit interval
$[0,1]$
by Cantor space
$\yocantor$
to obtain an isometric family of
ultrametrics with a prescribed range.
We then construct an ultrametric extension operator and use it to prove that the space of compatible ultrametrics with the prescribed range on a nonseparable ultrametrizable space,
as well as each of its components,
is non-Borel.

Fix a characteristic range set
$R$.

\subsection{Non-Borelness on discrete spaces}\label{subsec:ultrametric-discrete-nonborel}
We construct ultrametrics on a discrete space that induce the discrete topology precisely when their parameter lies outside a fixed non-Borel subset of Cantor space
$\yocantor$.

Choose a strictly decreasing map
$\youeta\colon\Z_{\ge 0}\to R\setminus\{0\}$
with
$\youeta(n)\to0$.
{For
$s,t\in\yocantor$,
put
\begin{equation}\label{eq:cantor-range-ultrametric}
h(s,t)=
\begin{cases}
\youeta(\min\{n\in\Z_{\ge 0}\mid s(n)\ne t(n)\}),&s\ne t,\\
0,&s=t.
\end{cases}
\end{equation}}
Then
$h$
is an
$R$-valued
ultrametric inducing the usual topology of
$\yocantor$.
Fix a dense set
$S\subseteq\yocantor$
supplied by
\Cref{prop:dense-aleph1-nonborel-cantor},
and fix a one-to-one enumeration
\begin{equation}\label{eq:ultrametric-coding-parameter-set}
S=\{w_\alpha\mid\alpha<\omega_1\}.
\end{equation}
For
$\alpha<\omega_1$,
put
\begin{equation}\label{eq:ultrametric-coding-cluster}
C_\alpha=\{y_\alpha\}\sqcup\{x_{\alpha,n}\mid n\in\Z_{\ge 0}\}.
\end{equation}
Put
\begin{equation}\label{eq:ultrametric-coding-underlying-set}
\yodis=\bigsqcup_{\alpha<\omega_1}C_\alpha.
\end{equation}
We equip
$\yodis$
with the discrete topology.

\begin{lemma}\label{lem:ultrametric-coding}
For
 $t\in\yocantor$,
 $\alpha<\omega_1$,
 and
  $n\in\Z_{\ge 0}$,
  put
\begin{equation}\label{eq:ultrametric-coding-length}
l_n(\alpha,t)=h(t,w_\alpha)\lor \youeta(n).
\end{equation}
{Fix
$t\in\yocantor$.
Set
$d_t(u,u)=0$
for every
$u\in\yodis$.
For distinct
$u,v\in\yodis$,
define
$d_t(u,v)$
symmetrically by the following formula with indices
$\alpha<\omega_1$
and
$m,n\in\Z_{\ge 0}$
}
\begin{equation}\label{eq:ultrametric-coding-definition}
d_t(u,v)=
\begin{cases}
l_n(\alpha,t) &\{u,v\}=\{y_\alpha,x_{\alpha,n}\},\\
l_n(\alpha,t)\lor l_m(\alpha,t)
  &\{u,v\}=\{x_{\alpha,n},x_{\alpha,m}\},\ n\ne m,\\
\youeta(0) &\text{otherwise}.
\end{cases}
\end{equation}
Then
 $d_t$
 is an
 $R$-valued
 ultrametric on the underlying set of
 $\yodis$,
 and the following statements hold.
\begin{enumerate}[label=\textup{(R\arabic*)},ref=\textup{(R\arabic*)}]
\item\label{item:ultrametric-coding-compatibility}
We have
$d_t\in\youltsp{\yodis}{R}$
if and only if
$t\in\yocantor\setminus S$.
\item\label{item:ultrametric-coding-completeness}
We have
$d_t\in\youcompm{\yodis}{R}$
if and only if
$t\in\yocantor\setminus S$.
\end{enumerate}
\end{lemma}

\begin{proof}
For each
$\alpha<\omega_1$,
formula \eqref{eq:ultrametric-coding-definition} shows that the restriction of
$d_t$
to
$C_\alpha^2$
satisfies the strong triangle inequality.
Distinct sets
$C_\alpha$
have distance
$\youeta(0)$,
and each
$C_\alpha$
has diameter at most
$\youeta(0)$.
Hence
$d_t$
is an
$R$-valued
ultrametric.
We now prove assertions
\ref{item:ultrametric-coding-compatibility} and
\ref{item:ultrametric-coding-completeness}.
We first prove assertion~\ref{item:ultrametric-coding-compatibility}.
If
$t\in \yocantor\setminus S$,
then for every
$\alpha<\omega_1$
and
$n\in\Z_{\ge 0}$,
we have
$\yoopenball{x_{\alpha, n}}{\youeta(n)}{d_t}=\{x_{\alpha, n}\}$
and
$\yoopenball{y_\alpha}{h(t, w_{\alpha})}{d_t}=\{y_\alpha\}$.
Thus
$d_t$
induces the discrete topology.
 If
 $t=w_\beta$
 for some
 $\beta\in \omega_1$,
 then
$d_t(y_\beta,x_{\beta,n})=\youeta(n)\to 0$
as
$n\to \infty$.
Thus
$y_\beta$
is not isolated.
Thus assertion~\ref{item:ultrametric-coding-compatibility} holds.

{By assertion~\ref{item:ultrametric-coding-compatibility} and
$\youcompm{\yodis}{R}\subseteq\youltsp{\yodis}{R}$,
for every
$t\in\yocantor$,
$d_t\in\youcompm{\yodis}{R}$
implies
$t\notin S$.}

Let
$t\in\yocantor\setminus S$,
and let
$\{u_n\}_{n\in\Z_{\ge 0}}$
be a Cauchy sequence in
$(\yodis,d_t)$.
Choose
$N_0\in\Z_{\ge 0}$
such that
$d_t(u_n,u_m)<\youeta(0)$
for all
$n,m\ge N_0$.
Let
$\alpha<\omega_1$
satisfy
$u_{N_0}\in C_\alpha$.
Since points in distinct sets
$C_\alpha$
have distance
$\youeta(0)$,
we obtain
$u_n\in C_\alpha$
for every
$n\ge N_0$.
Since
$t\notin S$,
we have
$h(t,w_\alpha)>0$,
and distinct points of
$C_\alpha$
have distance at least
$h(t,w_\alpha)$.
Choose
$N_1\ge N_0$
such that
$d_t(u_n,u_m)<h(t,w_\alpha)$
for all
$n,m\ge N_1$.
Then
$u_n=u_{N_1}$
for every
$n\ge N_1$.
Thus the sequence is eventually constant and converges in
$(\yodis,d_t)$.
Therefore
$d_t$
is complete.
This proves assertion~\ref{item:ultrametric-coding-completeness}.
\end{proof}

\begin{lemma}
\label{lem:closed-ultrametric-parameter-embedding}
The map
$\Phi\colon\yocantor\to\youcpm{\yodis}{R}$
defined by
$\Phi(t)=d_t$
is an isometric embedding,
and
the image
$\Phi(\yocantor \setminus S)$
is closed in
$\youltsp{\yodis}{R}$.
\end{lemma}

\begin{proof}
Fix
$s,t\in\yocantor$
and put
$\rho=h(s,t)$.
For every
$\alpha<\omega_1$,
the strong triangle inequality for
$h$
yields
\[
h(s,w_\alpha)\le h(t,w_\alpha)\lor\rho,
\qquad
h(t,w_\alpha)\le h(s,w_\alpha)\lor\rho.
\]
Consequently,
for every
$\alpha<\omega_1$
and
$n\in\Z_{\ge 0}$,
we have
\[
l_n(\alpha,s)\le l_n(\alpha,t)\lor\rho,
\qquad
l_n(\alpha,t)\le l_n(\alpha,s)\lor\rho.
\]
It follows from the definition of
$d_s$
and
$d_t$
that,
for all
$u,v\in\yodis$,
\[
d_s(u,v)\le d_t(u,v)\lor\rho,
\qquad
d_t(u,v)\le d_s(u,v)\lor\rho.
\]
Thus
$\rho\in\youadm_{\yodis}^R(d_s,d_t)$
by
\eqref{eq:ultrametric-admissible-threshold-set},
and hence,
by
\eqref{eq:non-archimedean-uniform-metric},
\[
\youmetdis_{\yodis}^R(d_s,d_t)\le\rho.
\]
To prove the opposite inequality,
assume
that
$s\ne t$.
By the density of
$S$,
choose
$w_\alpha$
such that
$h(s,w_\alpha)<\rho$,
and then choose
$n$
such that
$\youeta(n)<\rho$.
The strong triangle inequality yields
$h(t,w_\alpha)\le\rho$.
On the other hand,
\[
\rho=h(s,t)\le h(s,w_\alpha)\lor h(w_\alpha,t),
\]
and therefore
$h(t,w_\alpha)=\rho$.
Hence
$d_s(y_\alpha,x_{\alpha,n})<\rho$
and
$d_t(y_\alpha,x_{\alpha,n})=\rho$.
For every
$\varepsilon\in R$
with
$\varepsilon<\rho$,
we then have
\[
d_t(y_\alpha,x_{\alpha,n})
>d_s(y_\alpha,x_{\alpha,n})\lor\varepsilon.
\]
Thus
$\varepsilon\notin\youadm_{\yodis}^R(d_s,d_t)$
for every
$\varepsilon\in R$
with
$\varepsilon<\rho$.
By \eqref{eq:non-archimedean-uniform-metric},
it follows that
$\youmetdis_{\yodis}^R(d_s,d_t)\ge\rho$,
and hence
$\youmetdis_{\yodis}^R(d_s,d_t)=h(s,t)$.

It remains to show that
$\Phi(\yocantor \setminus S)$
is closed.
Property~\ref{item:ultrametric-coding-compatibility} of
\Cref{lem:ultrametric-coding} yields
\begin{equation}\label{eq:ultrametric-parameter-compatible-pullback}
\Phi^{-1}(\youltsp{\yodis}{R})=\yocantor\setminus S.
\end{equation}
Since
$\yocantor$
is compact and
$\youcpm{\yodis}{R}$
is Hausdorff,
the
second assertion of \Cref{lem:continuous-pullback-obstruction} shows that
$\Phi(\yocantor\setminus S)$
is closed in
$\youltsp{\yodis}{R}$.
\end{proof}

\begin{theorem}
\label{thm:discrete-ultrametric-nonborel}
Let
$\yodis$
be a discrete space of cardinality
$\aleph_1$.
Then
$\youltsp{\yodis}{R}$
is non-Borel in
$\youcpm{\yodis}{R}$.
\end{theorem}

\begin{proof}
By \eqref{eq:ultrametric-parameter-compatible-pullback} and
\Cref{lem:continuous-pullback-obstruction,prop:dense-aleph1-nonborel-cantor},
$\youltsp{\yodis}{R}$
is non-Borel in
$\youcpm{\yodis}{R}$.
\end{proof}

\begin{theorem}
\label{thm:ultrametric-component-nonborel}
Let
$\yodis$
be a discrete space of cardinality
$\aleph_1$.
Then every
$R$-component
$\youccl\in\youslamet{\yodis}{R}$
is
non-Borel in
$\youcpm{\yodis}{R}$.
\end{theorem}

\begin{proof}
We argue as in \Cref{thm:metric-component-nonborel}.
Fix
$d\in\youccl$
and
$z\in \yodis$.
Since
$R\subseteq[0,\infty)$
 has a countable cofinal subset,
choose
$q\in R$
such that
$H=\yoclosedball{z}{q}{d}$
has cardinality
$\aleph_1$.
Let
$\yocodesp$
be a copy of the discrete space in
\Cref{lem:ultrametric-coding},
with its metrics
$d_t$
and distinguished point
$y_0$.
Choose a bijection
$\theta\colon\yocodesp\to H$
with
$\theta(y_0)=z$.
For
$t\in\yocantor$
and
$a,b\in H$,
put
\begin{equation}\label{eq:ultrametric-transferred-coding}
e_t(a,b)=d_t(\theta^{-1}(a),\theta^{-1}(b)).
\end{equation}
For each
$t\in\yocantor$,
amalgamate
$e_t$
with
$d$
on
$(\yodis\setminus H)\cup\{z\}$
by putting,
for
$x\in\yodis\setminus H$
and
$a\in H$,
\begin{equation}\label{eq:discrete-ultrametric-cross-pair}
m_t(x,a)=m_t(a,x)=d(x,z)\lor e_t(z,a).
\end{equation}
By \Cref{lem:one-point-amalgamation},
$m_t$
is an ultrametric.
For
$x\in\yodis\setminus H$
and
$a\in H$,
\eqref{eq:discrete-ultrametric-cross-pair} yields
$m_t(x,a)\ge d(x,z)>q$.
Thus compatibility of
$m_t$
with the discrete topology is equivalent to compatibility of
$e_t$
on
$H$.
By \Cref{lem:ultrametric-coding} and
\eqref{eq:ultrametric-transferred-coding},
\begin{equation}\label{eq:discrete-ultrametric-amalgamation-compatibility}
m_t\in\youltsp{\yodis}{R}
\quad\Longleftrightarrow\quad
t\in\yocantor\setminus S.
\end{equation}

{Fix
$t\in\yocantor$
and put
$\varepsilon=q\lor\youeta(0)$.
On
$H^2$,
the metrics
$d$
and
$m_t=e_t$
are bounded by
$q$
and
$\youeta(0)$,
respectively.
On
$((\yodis\setminus H)\cup\{z\})^2$,
they agree.
For
$x\in\yodis\setminus H$
and
$a\in H$,
the strong triangle inequality yields
$d(x,a)=d(x,z)>q$,
and \eqref{eq:discrete-ultrametric-cross-pair} therefore implies
\[
d(x,a)\le m_t(x,a)=d(x,a)\lor e_t(z,a)\le d(x,a)\lor\varepsilon.
\]
Hence
$\varepsilon\in\youadm_{\yodis}^R(m_t,d)$,
and consequently
}
\begin{equation}\label{eq:discrete-ultrametric-amalgamation-component-bound}
\youmetdis_{\yodis}^R(m_t,d)\le\varepsilon.
\end{equation}
For
$s,t\in\yocantor$,
\eqref{eq:discrete-ultrametric-cross-pair} shows that every admissible threshold
for
$e_s$
and
$e_t$
is also admissible for
$m_s$
and
$m_t$.
Restriction to
$H^2$
proves the reverse implication.
{Hence,
for all
$s,t\in\yocantor$,
\Cref{lem:closed-ultrametric-parameter-embedding} and
\eqref{eq:ultrametric-transferred-coding} imply}
\begin{equation}\label{eq:discrete-ultrametric-amalgamation-isometry}
\youmetdis_{\yodis}^R(m_s,m_t)=h(s,t).
\end{equation}
Define
$\Psi\colon\yocantor\to\youcpm{\yodis}{R}$
by
$\Psi(t)=m_t$.
Equations \eqref{eq:discrete-ultrametric-amalgamation-compatibility},
\eqref{eq:discrete-ultrametric-amalgamation-component-bound},
and
\eqref{eq:discrete-ultrametric-amalgamation-isometry} yield
\begin{equation}\label{eq:ultrametric-component-coding-pullback}
\Psi^{-1}(\youccl)=\yocantor\setminus S.
\end{equation}
Thus the conclusion follows from
\Cref{lem:continuous-pullback-obstruction,prop:dense-aleph1-nonborel-cantor}.
\end{proof}

\subsection{Extension to general ultrametrizable spaces}\label{subsec:ultrametric-extension-general}
We now extend ultrametrics from a closed discrete net to the whole space.
The extension operator is chosen so that the restriction of a fixed compatible ultrametric is extended to a member of the same
$R$-component.

\begin{lemma}\label{lem:ultrametric-extensor-preserving-component}
Let
$(Y,d)$
be a nonempty
$R$-valued ultrametric space,
let
$r\in R\setminus\{0\}$,
and let
$\youclnet$
be a maximal
$r$-separated
subset of
$(Y,d)$.
Then there exist an isometric extensor
\[
E\colon\youltsp{\youclnet}{R}\to\youltsp{Y}{R}
\]
and
$\tau\in R$
such that
\[
\youmetdis_Y^R(E(d|_{\youclnet^2}),d)\le\tau<\infty.
\]
In particular,
$E(d|_{\youclnet^2})$
and
$d$
belong to the same
$R$-component of
$\youltsp{Y}{R}$.
\end{lemma}

\begin{proof}
We make extensive use of the constructions and arguments in \cite{Ishiki2023Factorization}.
The separation condition shows that
$\youclnet$
is closed and discrete in
$Y$.
By \Cref{lem:maximal-separated-set-ultrametric},
the family
$\{\yoopenball{a}{r}{d}\}_{a\in\youclnet}$
is a partition of
$Y$
into clopen subsets.
{Define
$\pi\colon Y\to\youclnet$
by requiring that,
for each
$x\in Y$
and the unique point
$a\in\youclnet$
with
$x\in\yoopenball{a}{r}{d}$,
\begin{equation}\label{eq:clopen-net-retraction-map}
\pi(x)=a.
\end{equation}}
\Cref{lem:maximal-separated-set-ultrametric}
shows that
$\pi$
is continuous and that,
for every
$x\in Y$,
\begin{equation}\label{eq:clopen-net-retraction}
d(x,\pi(x))<r.
\end{equation}

Choose a point
$\youquotpt$
not belonging to
$Y$
and put
\[
Y_{/\youclnet}=(Y\setminus\youclnet)\sqcup\{\youquotpt\}.
\]
Define
$p\colon Y\to Y_{/\youclnet}$
by
\[
p(x)=
\begin{cases}
\youquotpt & x\in\youclnet,\\
x & x\in Y\setminus\youclnet.
\end{cases}
\]
Thus
$p$
identifies all points of
$\youclnet$
with the single point
$\youquotpt$.
For
$x\in Y$,
put
$\rho(x)=\inf_{a\in\youclnet}d(x,a)$.
We equip
$Y_{/\youclnet}$
with the metric
$\delta$
defined,
for
$x,y\in Y$,
by
\begin{equation}\label{eq:net-metric-quotient}
\delta(p(x),p(y))=\min\{d(x,y),\rho(x)+\rho(y)\}.
\end{equation}
This is the metric quotient of
$(Y,d)$
along
$\youclnet$
in
\cite[Proposition~2.6]{Ishiki2023Factorization}.
In particular,
$p$
is continuous,
and
$\delta(p(x),\youquotpt)=\rho(x)$.
{If
$\youclnet=Y$,
then
$Y_{/\youclnet}$
is a singleton and has covering dimension zero.
If
$\youclnet\ne Y$,
then
$Y\setminus\youclnet$
is a nonempty ultrametrizable space and has covering dimension zero.
In this case,
\cite[Proposition~2.8]{Ishiki2023Factorization} shows that
$Y_{/\youclnet}$
also has covering dimension zero.
Since
$R$
is characteristic,
\cite[Proposition~2.3]{Ishiki2023Factorization} provides a compatible
$R$-valued
ultrametric
$v_0$
on
$Y_{/\youclnet}$.}
Choose
$b\in R\setminus\{0\}$.
For
$x,y\in Y_{/\youclnet}$,
put
$v(x,y)=\min\{v_0(x,y),b\}$.
Then
$v$
satisfies
$v\le b$
and it generates the topology of
$Y_{/\youclnet}$.

We now prove that
for each
$x\in Y$,
we have
\begin{equation}\label{eq:net-nearest-point}
d(x,\pi(x))=\rho(x).
\end{equation}
Inequality \eqref{eq:clopen-net-retraction} shows that this distance is less than
$r$,
whereas for
$a\in\youclnet\setminus\{\pi(x)\}$,
the separation condition and the strong triangle inequality yield
$d(x,a)=d(\pi(x),a)\ge r$.
Thus we obtain
\eqref{eq:net-nearest-point}.

The map
$\Theta\colon Y\to\youclnet\times Y_{/\youclnet}$
defined by
$\Theta(x)=(\pi(x),p(x))$
is continuous and injective.
We verify that it is a topological embedding.
The restriction of
$p$
to
$Y\setminus\youclnet$
is a homeomorphism onto
$Y_{/\youclnet}\setminus\{\youquotpt\}$
by
\cite[Proposition~2.6]{Ishiki2023Factorization}.
At a point
$a\in\youclnet$,
for every
$s>0$,
\eqref{eq:net-nearest-point} yields
\[
\Theta^{-1}\!\left(\{a\}\times
\yoopenball{\youquotpt}{s}{\delta}\right)
\subseteq\yoopenball{a}{s}{d}.
\]
Since
$\{a\}$
is open in
$\youclnet$,
these neighborhoods prove continuity of the inverse at
$\Theta(a)$.
Following
\cite[proof of Theorem~4.7]{Ishiki2023Factorization},
for
$w\in\youltsp{\youclnet}{R}$
and
$x,y\in Y$,
define
\begin{equation}\label{eq:component-controlled-ultrametric-extension}
E(w)(x,y)
=w(\pi(x),\pi(y))\lor v(p(x),p(y)).
\end{equation}
The cited proof shows that
$E$
is an isometric extensor.
Thus
for every
$w\in\youltsp{\youclnet}{R}$,
we have
$E(w)|_{\youclnet^2}=w$.
For
$w_1,w_2\in\youltsp{\youclnet}{R}$,
we have
\begin{equation}\label{eq:component-controlled-extension-isometry}
\youmetdis_Y^R(E(w_1),E(w_2))
=\youmetdis_{\youclnet}^R(w_1,w_2).
\end{equation}

Put
$e=d|_{\youclnet^2}$
and
$\widehat d=E(e)$.
For every
$x\in Y$,
\eqref{eq:component-controlled-ultrametric-extension} yields
$\widehat d(x,\pi(x))\le b$.
Hence,
with
$\tau=r\lor b$,
for all
$x,y\in Y$,
the strong triangle inequality and \eqref{eq:clopen-net-retraction} yield
\begin{equation}\label{eq:component-controlled-base-comparison}
\widehat d(x,y)\le d(x,y)\lor\tau,
\qquad
d(x,y)\le \widehat d(x,y)\lor\tau.
\end{equation}
Consequently,
\begin{equation}\label{eq:base-ultrametric-component-bound}
\youmetdis_Y^R(\widehat d,d)\le\tau<\infty.
\end{equation}
\end{proof}

{The ultrametrics constructed above with parameters in
$\yocantor\setminus S$
generate the discrete topology.
We apply the extensor in
\Cref{lem:ultrametric-extensor-preserving-component}
to these ultrametrics.
The estimate in
\Cref{lem:ultrametric-extensor-preserving-component}
ensures that their extensions belong to the prescribed
$R$-component.}

\begin{theorem}\label{cor:nonseparable-ultrametric-nonborel}
Let
$Y$
be a nonseparable ultrametrizable space.
Then the following
statements hold.
\begin{enumerate}[label=\textup{(U\arabic*)},ref=\textup{(U\arabic*)}]
\item\label{item:nonseparable-ultrametric-nonborel-components}
Every
$R$-component
of
$\youltsp{Y}{R}$
is non-Borel in
$\youcpm{Y}{R}$.
\item\label{item:nonseparable-ultrametric-nonborel-space}
The space
$\youltsp{Y}{R}$
is non-Borel in
$\youcpm{Y}{R}$.
\end{enumerate}
\end{theorem}

\begin{proof}
We first prove assertion~\ref{item:nonseparable-ultrametric-nonborel-components}.
Fix an
$R$-component
$\youccl$
of
$\youltsp{Y}{R}$
and a member
$d\in\youccl$.
By \Cref{lem:maximal-separated-set},
there is an
$\varepsilon>0$
for which
$(Y,d)$
has no countable
$\varepsilon$-net.
Choose
$r\in R\setminus\{0\}$
with
$r\le\varepsilon$,
and let
$\youclnet$
be a maximal
$r$-separated
subset of
$(Y,d)$.
By \Cref{lem:maximal-separated-set-ultrametric},
$\youclnet$
is a closed discrete
$r$-net.
It is uncountable because
$r\le\varepsilon$.
Put
$e=d|_{\youclnet^2}$.
Apply \Cref{lem:ultrametric-extensor-preserving-component}
to obtain an isometric extensor
$E\colon\youltsp{\youclnet}{R}\to\youltsp{Y}{R}$
and
$\tau\in R$
satisfying
\eqref{eq:base-ultrametric-component-bound},
where
$\widehat d=E(e)$.

We now construct a family of ultrametrics on
$\youclnet$.
Fix
$z\in\youclnet$.
Since
$R$
has a countable cofinal subset,
there is
$q\in R$
such that the following closed ball is uncountable.
\begin{equation}\label{eq:uncountable-net-ball}
B=\yoclosedball{z}{q}{e}.
\end{equation}
Choose
$H\subseteq B$
with
$z\in H$
and
$\card(H)=\aleph_1$.
Let
$\yocodesp$
be a copy of the discrete space in
\Cref{lem:ultrametric-coding},
with its metrics
$d_t$
and distinguished point
$y_0$.
Choose a bijection
$\theta\colon\yocodesp\to H$
with
$\theta(y_0)=z$.
For
$t\in\yocantor$
and
$a,b\in H$,
put
\begin{equation}\label{eq:net-transferred-coding}
e_t(a,b)=d_t(\theta^{-1}(a),\theta^{-1}(b)).
\end{equation}
For each
$t\in\yocantor$,
amalgamate
$e_t$
with
$e$
on
$(\youclnet\setminus H)\cup\{z\}$
by putting,
for
$x\in\youclnet\setminus H$
and
$u\in H$,
\begin{equation}\label{eq:net-coding-amalgamation}
m_t(x,u)=m_t(u,x)=e(x,z)\lor e_t(z,u).
\end{equation}
Fix
$t\in\yocantor$.
By
\Cref{lem:one-point-amalgamation},
$m_t$
is an
$R$-valued
ultrametric on
$\youclnet$.
For
$x\in\youclnet\setminus H$
and
$u\in H$,
\eqref{eq:net-coding-amalgamation} and the
$r$-separation of
$\youclnet$
imply
\[
m_t(x,u)\ge e(x,z)\ge r.
\]
Since
$m_t=e$
on
$((\youclnet\setminus H)\cup\{z\})^2$,
compatibility of
$m_t$
with the discrete topology on
$\youclnet$
is equivalent to compatibility of
$e_t$
with the discrete topology on
$H$.
By property~\ref{item:ultrametric-coding-compatibility} of
\Cref{lem:ultrametric-coding} and
\eqref{eq:net-transferred-coding},
we therefore obtain
\begin{equation}\label{eq:amalgamated-coding-compatibility}
m_t\in\youltsp{\youclnet}{R}
\quad\Longleftrightarrow\quad
t\in\yocantor\setminus S.
\end{equation}

Put
$\gamma=q\lor \youeta(0)$.
On
$H^2$,
both
$e$
and
$m_t$
are
bounded by
$\gamma$.
On
$((\youclnet\setminus H)\cup\{z\})^2$
they agree.
For
$x\in\youclnet\setminus H$
and
$u\in H$,
\eqref{eq:net-coding-amalgamation} and the strong triangle inequality yield
\[
\begin{aligned}
m_t(x,u)&\le e(x,u)\lor e(u,z)\lor e_t(z,u)\\
&\le e(x,u)\lor\gamma,\\
e(x,u)&\le e(x,z)\lor e(z,u)\\
&\le m_t(x,u)\lor\gamma.
\end{aligned}
\]
Thus
$\gamma\in\youadm_{\youclnet}^R(m_t,e)$,
and therefore
\begin{equation}\label{eq:amalgamated-coding-component-bound}
\youmetdis_{\youclnet}^R(m_t,e)\le\gamma.
\end{equation}
For
$s,t\in\yocantor$,
\eqref{eq:net-coding-amalgamation} and restriction to
$H^2$
yield
\begin{equation}\label{eq:net-coding-thresholds}
\youadm_{\youclnet}^R(m_s,m_t)=\youadm_H^R(e_s,e_t).
\end{equation}
{For all
$s,t\in\yocantor$,
combining} \eqref{eq:net-coding-thresholds} with
\eqref{eq:net-transferred-coding} and
\Cref{lem:closed-ultrametric-parameter-embedding},
we obtain
\begin{equation}\label{eq:amalgamated-coding-isometry}
\youmetdis_{\youclnet}^R(m_s,m_t)=h(s,t).
\end{equation}

By
\Cref{prop:dense-aleph1-nonborel-cantor},
$\yocantor\setminus S$
is dense in
$\yocantor$.
For
$t\in\yocantor\setminus S$,
define
\begin{equation}\label{eq:component-controlled-parameter-map}
F(t)=E(m_t).
\end{equation}
By \eqref{eq:component-controlled-extension-isometry} and
\eqref{eq:amalgamated-coding-isometry},
the map
$F$
is isometric.

By \Cref{lem:ucpm-complete},
the space
$\youcpm{Y}{R}$
is complete with respect to
$\youmetdis_Y^R$.

Since
$\yocantor\setminus S$
is dense in
$\yocantor$,
the isometry
$F$
has a unique isometric
extension
\begin{equation}\label{eq:extended-ultrametric-parameter-map}
\overline F\colon \yocantor\to\youcpm{Y}{R}.
\end{equation}
The maps
$t\mapsto\overline F(t)|_{\youclnet^2}$
and
$t\mapsto m_t$
from
$\yocantor$
to
$\youcpm{\youclnet}{R}$
are continuous by continuity of restriction and
\eqref{eq:amalgamated-coding-isometry}.
They agree on the dense set
$\yocantor\setminus S$
since
$E$
is an extension operator.
Consequently,
for every
$t\in\yocantor$,
\begin{equation}\label{eq:extended-ultrametric-coding-restriction}
\overline F(t)|_{\youclnet^2}=m_t.
\end{equation}
{For every
$t\in\yocantor$,
it follows from
\eqref{eq:extended-ultrametric-coding-restriction} and
\eqref{eq:amalgamated-coding-compatibility} that}
\begin{equation}\label{eq:extended-ultrametric-coding-compatibility}
\overline F(t)\in\youltsp{Y}{R}
\quad\Longleftrightarrow\quad
t\in\yocantor\setminus S.
\end{equation}
Here the forward implication follows by restricting a compatible
ultrametric to the closed discrete subspace
$\youclnet$,
and the reverse
implication follows from
$\overline F(t)=E(m_t)$
for
$t\in\yocantor\setminus S$.

For
$t\in\yocantor\setminus S$,
the isometry of
$E$
together with
\eqref{eq:base-ultrametric-component-bound} and
\eqref{eq:amalgamated-coding-component-bound} yields
\begin{equation}\label{eq:extended-ultrametric-coding-component-bound}
\youmetdis_Y^R(\overline F(t),d)\le\gamma\lor\tau<\infty.
\end{equation}
Thus
$\overline F(\yocantor\setminus S)\subseteq\youccl$.
Equations \eqref{eq:extended-ultrametric-coding-compatibility} and
\eqref{eq:extended-ultrametric-coding-component-bound} yield
\begin{equation}\label{eq:extended-ultrametric-component-pullback}
\overline F^{-1}(\youccl)=\yocantor\setminus S.
\end{equation}
Thus \Cref{lem:continuous-pullback-obstruction} proves
assertion~\ref{item:nonseparable-ultrametric-nonborel-components}.
{Since
$Y$
is ultrametrizable and
$R$
is characteristic,
\cite[Proposition~2.3]{Ishiki2023Factorization} implies that
$\youltsp{Y}{R}\ne\emptyset$.}

Every
$R$-component
is open in
$\youltsp{Y}{R}$,
so
assertion~\ref{item:nonseparable-ultrametric-nonborel-space}
follows from
assertion~\ref{item:nonseparable-ultrametric-nonborel-components}
and \Cref{lem:relative-borel-inheritance}.
\end{proof}

\subsection{Borelness and separability}\label{subsec:ultrametric-borel-separability}
The preceding result implies separability of the underlying space whenever the space of compatible ultrametrics is Borel or completely metrizable.

\begin{theorem}
\label{thm:ultrametric-borel-separability}
Let
$Y$
be an ultrametrizable space.
If
$\youltsp{Y}{R}$
is
Borel
in
$\youcpm{Y}{R}$,
then
$Y$
is separable.
\end{theorem}

\begin{proof}
For the sake of contradiction,
suppose that
$Y$
is nonseparable.
\Cref{cor:nonseparable-ultrametric-nonborel} would imply that
$\youltsp{Y}{R}$
is non-Borel in
$\youcpm{Y}{R}$,
contrary to the
hypothesis.
\end{proof}

\begin{corollary}\label{cor:ultrametric-space-complete-separability}
Let
$Y$
be an ultrametrizable space.
If
$\youltsp{Y}{R}$
is completely
metrizable,
then
$Y$
is separable.
\end{corollary}

\begin{proof}
If
$\youltsp{Y}{R}$
is completely metrizable,
then by
\Cref{thm:complete-metrizability-g-delta},
it is a
$G_\delta$,
and hence Borel,
subspace of
$\youcpm{Y}{R}$.
By
\Cref{thm:ultrametric-borel-separability},
$Y$
is separable.
\end{proof}

\subsection{Borelness of spaces of complete ultrametrics}\label{subsec:complete-ultrametric-nonborel}
We next study Borelness of the space of complete ultrametrics.
After treating discrete spaces,
we use extension operators that preserve completeness to pass to general ultrametrizable spaces.

\begin{theorem}
\label{thm:complete-ultrametric-nonborel}
Let
$\yodis$
be a discrete space of cardinality
$\aleph_1$.
Then for every
$\youccl\in\youslamet{\yodis}{R}$,
the space
$\youccl\cap\youcompm{\yodis}{R}$
is non-Borel in
$\youcpm{\yodis}{R}$.
 Consequently,
 $\youcompm{\yodis}{R}$
 is also non-Borel.
\end{theorem}

\begin{proof}
Choose
$d\in\youccl$
and
$c\in R\setminus\{0\}$.
For
$x,y\in\yodis$,
define
\[
\yobard(x,y)=
\begin{cases}
0 & x=y,\\
d(x,y)\lor c & x\ne y.
\end{cases}
\]
Then
$\yobard$
is a uniformly discrete complete ultrametric,
and
$\youmetdis_{\yodis}^R(d,\yobard)\le c$
implies that
$\yobard\in\youccl$.

{Repeat the construction in
\Cref{thm:ultrametric-component-nonborel}
with
$\yobard$
as the base metric,
using \eqref{eq:ultrametric-transferred-coding} to define the metrics
$e_t$
on
$H$
and \eqref{eq:discrete-ultrametric-cross-pair} to define
$m_t$
on
$\yodis$.
Fix
$t\in\yocantor\setminus S$.
By \eqref{eq:ultrametric-transferred-coding},
the bijection
$\theta$
is an isometry from
$(\yocodesp,d_t)$
to
$(H,e_t)$.
Thus
$e_t$
is complete by property~\ref{item:ultrametric-coding-completeness} of
\Cref{lem:ultrametric-coding}.
}
The restriction of
$\yobard$
to
$(\yodis\setminus H)\cup\{z\}$
is complete because it is uniformly discrete.
Thus \Cref{lem:one-point-amalgamation} implies that
$m_t$
is complete.
Hence
\eqref{eq:discrete-ultrametric-amalgamation-compatibility},
\eqref{eq:discrete-ultrametric-amalgamation-component-bound},
and
\eqref{eq:discrete-ultrametric-amalgamation-isometry}
show that the parameter map
$\Psi\colon\yocantor\to\youcpm{\yodis}{R}$
defined by
$\Psi(t)=m_t$
satisfies
\begin{equation}\label{eq:complete-ultrametric-coding-pullback}
\Psi^{-1}(\youccl\cap\youcompm{\yodis}{R})=\yocantor\setminus S.
\end{equation}
Thus \eqref{eq:complete-ultrametric-coding-pullback} and
\Cref{lem:continuous-pullback-obstruction} show that
$\youccl\cap\youcompm{\yodis}{R}$
is non-Borel in
$\youcpm{\yodis}{R}$.

The
$R$-component of
$\youcpm{\yodis}{R}$
containing
$d$
is
clopen,
and its intersection with
$\youcompm{\yodis}{R}$
equals
$\youccl\cap\youcompm{\yodis}{R}$.
Hence
\Cref{lem:relative-borel-inheritance} shows that
$\youcompm{\yodis}{R}$
is also non-Borel.
\end{proof}

\begin{corollary}\label{cor:complete-ultrametric-borel-separability}
Let
$Y$
be an ultrametrizable space such that
$\youcompm{Y}{R}\ne\emptyset$.
If
$\youcompm{Y}{R}$
is
Borel in
$\youcpm{Y}{R}$,
then
$Y$
is separable.
The same conclusion
holds if
$\youcompm{Y}{R}$
is completely metrizable.
\end{corollary}

\begin{proof}
By \Cref{thm:complete-metrizability-g-delta},
complete metrizability of
$\youcompm{Y}{R}$
implies that it is a
$G_\delta$,
and hence Borel,
subspace
of
$\youcpm{Y}{R}$.
Thus it suffices to consider the Borel case.
For the sake of contradiction,
suppose that
$Y$
is nonseparable.
Choose
$d\in\youcompm{Y}{R}$.

{Repeat the construction in the proof of
\Cref{cor:nonseparable-ultrametric-nonborel}
with this
$d$
through the definition of
$m_t$
in \eqref{eq:net-coding-amalgamation}.
In particular,
let
$r$,
$\youclnet$,
$H$,
$z$,
and
$e=d|_{\youclnet^2}$
be as in that construction.
Fix
$t\in\yocantor\setminus S$.
By \eqref{eq:net-transferred-coding},
the bijection
$\theta$
is an isometry from
$(\yocodesp,d_t)$
to
$(H,e_t)$.
Thus
$e_t$
is complete by property~\ref{item:ultrametric-coding-completeness} of
\Cref{lem:ultrametric-coding}.
Since
$\youclnet$
is
$r$-separated with respect to
$e$,
the restriction
$e|_{((\youclnet\setminus H)\cup\{z\})^2}$
is complete.
The metric
$m_t$
is the amalgamation of
$e_t$
and this restriction defined by \eqref{eq:net-coding-amalgamation},
so it is complete by \Cref{lem:one-point-amalgamation}.
}

The completeness of
$d$
implies that
$Y$
is completely metrizable.
Instead of the explicit extensor used to control the component in
\Cref{cor:nonseparable-ultrametric-nonborel},
choose an isometric extensor
that preserves completeness,
\begin{equation}\label{eq:complete-ultrametric-extensor}
E\colon\youltsp{\youclnet}{R}\to\youltsp{Y}{R}
\end{equation}
by property~\ref{item:ultrametric-extensor-completeness} of
\Cref{thm:ultrametric-isometric-extensor}.
For
$t\in\yocantor\setminus S$,
put
$F(t)=E(m_t)$.
Equation \eqref{eq:amalgamated-coding-isometry} and the
isometry of
$E$
show that
$F$
is isometric.
Since
$\yocantor\setminus S$
is dense in
$\yocantor$
and
$\youcpm{Y}{R}$
is complete with respect to
$\youmetdis_Y^R$
by
\Cref{lem:ucpm-complete},
the map
$F$
has a unique isometric extension
\begin{equation}\label{eq:complete-ultrametric-parameter-map}
\overline F\colon \yocantor\to\youcpm{Y}{R}.
\end{equation}
As in \eqref{eq:extended-ultrametric-coding-restriction},
continuity and agreement on the dense set
$\yocantor\setminus S$
imply that,
for every
$t\in\yocantor$,
\begin{equation}\label{eq:complete-ultrametric-coding-restriction}
\overline F(t)|_{\youclnet^2}=m_t.
\end{equation}
We now prove the following equality.
\begin{equation}\label{eq:ambient-complete-ultrametric-pullback}
\overline F^{-1}(\youcompm{Y}{R})=\yocantor\setminus S.
\end{equation}
Take
$t\in\yocantor\setminus S$.
Since
$m_t\in\youcompm{\youclnet}{R}$,
completeness preservation yields
$\overline F(t)=E(m_t)\in\youcompm{Y}{R}$.
Thus
$t\in\overline F^{-1}(\youcompm{Y}{R})$.
Conversely,
take
$t\in\overline F^{-1}(\youcompm{Y}{R})$.
By \eqref{eq:complete-ultrametric-coding-restriction},
the ultrametric
$m_t$
is compatible on
$\youclnet$.

Equation \eqref{eq:amalgamated-coding-compatibility} therefore implies that
$t\in\yocantor\setminus S$.
Equation \eqref{eq:ambient-complete-ultrametric-pullback} and
\Cref{lem:continuous-pullback-obstruction} yield a contradiction.
\end{proof}

\printbibliography

\end{document}